\documentclass[a4paper,runningheads,envcountsame,envcountsect]{llncs}

\usepackage[utf8]{inputenc}
\usepackage{amsmath}
\usepackage{amssymb}
\usepackage{graphicx}
\usepackage{subcaption}
\usepackage{pgfplots}

\pgfplotsset{compat=1.18}

\usepackage[
	backref,
	backend=biber,
	style=numeric,
	natbib=true,
    maxnames=99,
    sortcites,
]{biblatex}
\usepackage{hyperref}
\usepackage{xcolor}

\definecolor{gray}{rgb}{.4,.4,.4}
\definecolor{webgreen}{rgb}{0,.5,0}
\definecolor{Maroon}{cmyk}{0, 0.87, 0.68, 0.32}
\definecolor{RoyalBlue}{cmyk}{1, 0.50, 0, 0}

\hypersetup{
    colorlinks,
    citecolor=webgreen,
    filecolor=RoyalBlue,
    linkcolor=RoyalBlue,
    urlcolor=RoyalBlue
}

\usepackage[capitalize,nameinlink]{cleveref}
\crefname{equation}{Eq.}{Eqs.}
\Crefname{equation}{Eq.}{Eqs.}
\crefname{appendix}{Appendix}{Appendices}
\Crefname{appendix}{Appendix}{Appendices}
\usepackage{etoolbox}
\AtEndEnvironment{proof}{\qed}

\title{Supersingularity and Superspeciality Verification of Abelian Surfaces}

\author{
	Maria Corte-Real Santos \inst{1} \and
	Gioella Lorenzon \inst{2} \and
	Krijn Reijnders \inst{2}
}

\institute{
	École Normale Supérieure de Lyon,
	CNRS\\
    \and
    COSIC, KU Leuven
}

\usepackage{xcolor}
\usepackage{colonequals}

\usepackage{amsmath}
\usepackage{amssymb}
\usepackage{mathtools}
\usepackage{csquotes}
\usepackage{xspace}

\usepackage{multirow}
\usepackage{siunitx}
\usepackage{booktabs}
\usepackage{makecell}
\usepackage{tikz-cd}
\usepackage{pgfplots}
\pgfplotsset{compat=1.18}
\usepackage{enumitem}
\usepackage{pgfplotstable}
\usepackage{booktabs}
\usepackage{siunitx}

\usepackage[noend]{algorithmic}
\usepackage[ruled]{algorithm}
\newtheorem{lemalg}[algorithm]{Algorithm}
\crefname{lemalg}{Algorithm}{Algorithms}

\renewcommand{\phi}{\varphi}

\newcommand{\algorithmstyle}[1]{\ensuremath{\mathsf{#1}}\xspace}

\newcommand{\IsSupersingular}{\algorithmstyle{IsSupersingular}}
\newcommand{\DEFIsSupersingular}{\IsSupersingular}

\newcommand{\J}{\mathcal{J}}

\newcommand{\zero}{\mathbf{0}}

\newcommand{\Fp}{\mathbb{F}_p}
\newcommand{\Fpt}{\mathbb{F}_{p^2}}
\newcommand{\Fpthree}{\mathbb{F}_{p^3}}
\newcommand{\Fpf}{\mathbb{F}_{p^4}}
\newcommand{\Fps}{\mathbb{F}_{p^6}}
\newcommand{\Fpk}{\mathbb{F}_{p^k}}
\newcommand{\Fpkt}{\mathbb{F}_{p^{2k}}}
\newcommand{\Fpbar}{\overline{\mathbb{F}}_p}
\newcommand{\Fq}{\mathbb{F}_q}
\newcommand{\Fqk}{\mathbb{F}_{q^k}}

\newcommand{\Z}{\mathbb{Z}}

\DeclareMathOperator{\End}{End}
\DeclareMathOperator{\Gal}{Gal}
\DeclareMathOperator{\Aut}{Aut}
\DeclareMathOperator{\RA}{RA}
\DeclareMathOperator{\Twists}{Twists}

\newcommand{\ZZ}{\ensuremath{\mathbb{Z}}}

\newcommand{\NN}{\ensuremath{\mathbb{N}}}

\newcommand{\kbar}{\overline{k}}
\newcommand{\rgets}{\xleftarrow{\$}}

\newcommand{\picks}{\xleftarrow{\$}}

\begin{document}

\maketitle

\begin{abstract}
        Supersingular abelian surfaces are essential in isogeny-based cryptography.
    Despite this, we have no efficient algorithm to verify if a given abelian surface is supersingular.
    In this work, we initiate this research topic by giving an efficient Monte Carlo algorithm to verify
    if an abelian surface over $\Fp$ is supersingular in $O(\log p)$ with negligible failure probability,
    and an efficient conclusive algorithm if the order is smooth.
    We derive this algorithm by a careful analysis on the structure of supersingular Jacobians over $\Fp$.
    Furthermore, we derive efficient algorithms to verify if an abelian variety of any dimension is minimal or maximal,
    and to verify if a Jacobian of any dimension is superspecial.
\end{abstract}


\section{Introduction}
\label{sec:intro}

Since the spectacular break of SIDH/SIKE~\cite{attack1,attack2,attack3}, higher-dimensional techniques have
revolutionized isogeny-based cryptography.
Modern protocols,
such as the signature schemes SQIsign~\cite{NISTPQC-ADD-R2:SQIsign25} and PRISM~\cite{prism}, the NIKE MIKE~\cite{mike}, or the KEM POKÉ~\cite{poke}
heavily rely on higher-dimensional isogenies for both performance and small sizes.
Abelian varieties of dimension $g > 1$ have therefore become fundamental objects in isogeny-based cryptography.
More precisely, such protocols rely on \emph{supersingular} principally polarized abelian varieties,
and often even the more restricted \emph{superspecial} principally polarized abelian varieties.
With recent developments, such as the module action framework~\cite{mike,mike2} or the Shimura class group
action on abelian surfaces~\cite{shimact}, validating supersingularity becomes crucial to security.
This can be seen as the higher-dimensional equivalent of public key validation in CSIDH~\cite{CSIDH},
where omitting key validation would introduce a significant attack vector for an adversary.

\paragraph{Supersingular elliptic curves.}
Supersingularity verification of elliptic curves has been studied extensively in both mathematical and cryptographic literature.
\textcite{sutherland_supersingularity} initiated this study by describing a Monte Carlo test
for elliptic curves over $\Fp$ and $\Fpt$: take a random point $P \in E(\Fq)$ and check
if $[p]P = \pm P$. The probability of a false positive is $O(q^{-\frac{1}{2}})$
and therefore suitable for almost all cryptographic purposes, where $q$ is large.

Nevertheless, such a test does not \emph{conclusively prove} that a curve is supersingular.
To do so, Sutherland describes an
algorithm based on the depth of isogeny volcanoes. In CSIDH~\cite{CSIDH}, the supersingularity test
is instead performed by verifying that the order of a random point is larger than $4\sqrt{p}$,
which, by the bound on width of the Hasse-Weil interval, proves that the order of the curve is $p+1$,
 and therefore the curve is supersingular.
 dCSIDH~\cite{dCSIDH} optimizes this approach
for primes $p$ where $2^f \mid p + 1$ with $f > \tfrac{1}{2}\log p$,
e.g., most primes used in current-day isogeny-based cryptography.
\textcite{doliskani} describes an algorithm based on polynomial identity testing to establish
supersingularity, which uses the fact that the $p$-th division polynomial $\psi_p(x)$ is precisely $\pm 1$
for supersingular curves, as these have $p$-rank 0.
\textcite{val-and-the-gang} give an in-depth performance comparison between these algorithms,
and reinterpret Doliskani's test using Montgomery arithmetic, for which \textcite{robert_cubical}
gives a reinterpretation as a self-pairing using cubical arithmetic. Lastly, in
\cite{pairings-CSIDH},
pairings are used to optimize the verification of the order of the curve, moving much of the computational
cost from curve arithmetic to field arithmetic.

\paragraph{Supersingular and superspecial abelian surfaces.}
Moving up in dimension, there is no single analogue of supersingularity. Instead, there are several inequivalent
generalizations organized by the Ekedahl--Oort stratification. As a leading example in dimension 2,
supersingularity for elliptic curves may be generalized as either supersingularity
of an abelian surface, or the more specific notion of superspeciality. The latter property is often the right
generalization for cryptographic purposes. Although literature exists which studies this
classification of abelian varieties in this stratification, we are not aware of literature addressing the
specific question of \emph{efficiently} verifying the supersingularity and superspeciality of abelian surfaces over a
finite field. To the best of our knowledge, the best known methods to verify such results depend on
point-counting algorithms or $L$-polynomial computations, which are prohibitively slow for the large primes used in isogeny-based cryptography.
In this work, we present an efficient algorithm for supersingularity testing of abelian surfaces over $\Fp$.
We furthermore study the question of minimality and maximality of abelian varieties, which allows us to efficiently verify superspeciality of Jacobians over $\Fp$ and $\Fpt$.

\paragraph{Key Example.}
We demonstrate the core idea of our simple and efficient probabilistic verification with the following key example, in which we verify minimality or maximality of an abelian variety. This property implies superspeciality, and therefore supersingularity \cite[Thm. 5.3]{max_abvars}.

\begin{lemalg}\label{alg:key-example}
    Let $A$ be an abelian surface over $\Fpt$.
    Let $P \rgets A(\Fpt)$.
    \begin{itemize}
        \item {\upshape If $[p + 1]P = \zero_A$, return \textbf{true}.
        \item Otherwise, return \textbf{false}.}
    \end{itemize}
\end{lemalg}
\noindent The algorithm for minimal abelian surfaces replaces $p + 1$ by $p-1$. The probability of
a false positive is $O(p^{-1})$, and is thus negligibly small for large primes.
This follows from what we call the \emph{width-bound argument},
used by Sutherland for elliptic curves~\cite{sutherland_supersingularity}.
In \Cref{thm:minmax-test}, we generalize this argument to abelian varieties.
However, this algorithm only works for minimal and maximal varieties, whereas
general supersingularity testing is much more complicated. Already for the cases
of abelian surfaces over $\Fp$, we require a deeper understanding
of the group structures of supersingular abelian surfaces and their twists.
Furthermore, the above algorithm can never \emph{conclusively prove} that an abelian surface is supersingular
or even superspecial, which requires additional understanding of the isotropic subgroups of abelian surfaces.

\subsubsection*{Contributions.}
We provide the first practical polynomial time supersingularity and superspeciality verification algorithms
for abelian surfaces over $\Fp$. We develop both probabilistic and conclusive
algorithms.
\begin{enumerate}
    \item In \Cref{sec:structure}, we analyze the structure of supersingular and superspecial Jacobians over $\Fp$,
    resulting in three key results:
    \begin{enumerate}
        \item \Cref{thm:supersingular-orders}, which describes the possible orders of supersingular Jacobians over $\Fpt$;
        \item \Cref{tab:jacobian-summary}, giving the group structures of supersingular Jacobians over $\Fp$;
        \item \Cref{thm:twisting superspecial}, which classifies superspecial Jacobians using the automorphism group of the associated genus-$2$ curve.
    \end{enumerate}
    These results underlie both the probabilistic and conclusive testing.
    \item In \Cref{sec:testing}, we derive algorithms for \emph{probabilistic verification} of extremality,
    supersingularity and superspeciality of abelian surfaces.
    \begin{enumerate}
        \item In \Cref{alg:minmax-test} and \Cref{thm:minmax-test},
        we generalize \Cref{alg:key-example}
        to verify minimality and maximality of an abelian variety of any dimension.
        \item Using \Cref{thm:twisting superspecial} we construct \Cref{alg:superspecial-test}, an efficient algorithm for testing superspeciality of Jacobians over $\Fp$. Its correctness and runtime are proved in \Cref{thm:superspecial-test}.
        We show a similar result for superspecial Jacobians over $\Fpt$ in \Cref{app:twisted superspecials}.
        \item In \Cref{alg:supersingularity_test} we give a $O(\log p)$
        Monte Carlo supersingularity test for abelian surfaces over $\Fp$, with all arithmetic performed  over $\Fp$. Its correctness and runtime are based on \Cref{thm:probabilistic-testing}.
    \end{enumerate}
    \item In \Cref{sec:conclusive}, we derive algorithms for \emph{conclusive verification} of supersingularity for
    abelian surfaces over $\Fp$, summarized in \Cref{cor:conclusive-testing-non-cyclic,cor:conclusive-testing-cyclic}.
    We furthermore analyze the factorization and discrete logarithm requirements for these algorithms,
    and specialize to a case of specific cryptographic interest, i.e. when $2^f \mid p + 1$,
    with $f$ close to $\log p$.
\end{enumerate}

\noindent We give an overview of all verification algorithms in this work in
\Cref{tab:prob-overview} (probabilistic tests) and \Cref{tab:conc-overview} (conclusive tests),
listing the setting in which each algorithm applies, its complexity, and either its failure probability or its
structural requirements.
We implemented these algorithms in \texttt{Magma}, and
benchmark their performance in \Cref{sec:results}.
The code is available
at the following repository,
\begin{center}
    \small \url{https://github.com/KULeuven-COSIC/supersingularity-and-superspeciality}.
\end{center}
\begin{table}[h]
  \centering
  \small
  \setlength{\tabcolsep}{5pt}
  \renewcommand{\arraystretch}{1.2}
  \caption{Overview of probabilistic (Monte Carlo) verification algorithms in this work. Complexities count field operations in the stated field, and $q = p$ or $p^2$.}
  \label{tab:prob-overview}
  \vspace{6pt}
  \begin{tabular}{@{}llllc@{}}
    \toprule
    \textbf{Algorithm} & \textbf{Verifies} & \textbf{Variety} & \textbf{Complexity} & \textbf{False pos.} \\
    \midrule
    \Cref{alg:key-example}           & maximality       & $A$$/ \Fpt$    & $O(\log p)$ & $O(p^{-1})$ \\
    \Cref{alg:minmax-test}           & maximality       & $A / \Fq$, any $g$  & $O(\log p)$ & $O(p^{-1})$ \\
    \Cref{alg:superspecial-test}     & superspeciality  & $\J / \Fq$, any $g$ & $O(g \log p)$ & $O((g-1)\,p^{-1})$ \\
    \Cref{alg:supersingularity_test} & supersingularity & $A$$/ \Fp$     & $O(\log p)$ & $O(p^{-1/2})$ \\
    \bottomrule
  \end{tabular}
\end{table}

\begin{table}[h]
  \centering
  \small
  \setlength{\tabcolsep}{4pt}
  \renewcommand{\arraystretch}{1.2}
  \caption{Overview of conclusive verification algorithms in this work. In the first three rows, $n = p^2 - bp + 1$ for $b \in \{-1, 0, 1\}$ and $n = p\mp 1$ for $b = \pm 2$, and $\omega(n)$ is the number of distinct prime factors of $n$. We assume the factorization of $n$ is known.
  Complexities count field operations in the stated field.}
  \label{tab:conc-overview}
  \vspace{6pt}
  \begin{tabular}{@{}lllll@{}}
    \toprule
    \textbf{Algorithm} & \textbf{Verifies} & \textbf{Setting} & \textbf{Complexity} & \textbf{Requirements} \\
    \midrule
    \Cref{cor:conclusive-testing-non-cyclic} & supersing. & $b = 2$ over $\Fp$
      & $O(\log\omega(n)\log p)$ & factorization of $n$ \\
    \Cref{cor:conclusive-testing-non-cyclic} & supersing. & $b = -2$ over $\Fp$
      & $O(\log\omega(n)\log p)$ & DLog in $\mu_{p+1}$ \\
    \Cref{cor:conclusive-testing-cyclic} & supersing. & $b \in \{-1,0,1\}$ over $\Fp$
      & $O(\log\omega(n)\log p)$ & works over $\Fpk$ \\
    \Cref{lem:two torsion test} & supersing. & $b = \pm 2$ over $\Fp$
      & $O(\log p)$ & $p = c \cdot 2^f \pm 1$  \\
    \Cref{alg:conclusive-extremal-surfaces} & maximality & $b = -2$ over $\Fpt$
      & $O(\log p)$ & DLog in $\mu_{p+1}$ \\
    \Cref{lem:fpt two torsion} & maximality & $b = -2$ over $\Fpt$
      & $O(\log p)$ & $p = c \cdot 2^f - 1$ \\
    \bottomrule
  \end{tabular}
\end{table}

Some of this work is phrased using Jacobians, representing principally polarized abelian surfaces that are neither products nor Weil restrictions. However, in practice we need fairly little machinery to apply our algorithms:
the probabilistic algorithms rely only on uniform random sampling and scalar multiplication, whereas the conclusive
algorithms additionally require pairing computations and discrete logarithms. Our conventions and the role of the principal polarization are detailed in \Cref{sec:prelim}.

\paragraph{Related Work.}
The task of verifying supersingularity or superspeciality in dimension $2$
fits into the much more general task of classifying curves
and principally polarized abelian varieties according to their $p$-rank. This is called the Ekedahl--Oort stratification~\cite{oort}, and has been widely studied (see, e.g.,~\cite{devala,moonen}).

In dimension $2$, one may directly detect supersingularity by computing the $L$-polynomial of the
abelian surface or, equivalently, the characteristic polynomial of Frobenius. Methods to do so often rely on generalizations
of Schoof-like algorithms to determine the number of points on a curve or variety.
Most directly applicable is the Schoof--Pila algorithm~\cite{pila},
and subsequent improvements~\cite{gaudry-schost,kieffer}.
Of particular interest to this work is the article by~\textcite{jia-shi} on lifting $L$-polynomials of genus-$2$ curves.
For superspeciality, the most well-known approach relies on computation of the Cartier-Manin or Hasse-Witt matrix, which also
allows one to verify supersingularity.
In \Cref{sec:results}, we compare these approaches to the algorithms given in this work.

\subsubsection{Acknowledgements.}
This work is supported by the European Research
Council under grant No. 101116169 (AGATHA CRYPTY),
under the European Union’s Horizon 2020 research and innovation programme
(grant agreement ISOCRYPT – No. 101020788), by the Research Council KU Leuven grant
C14/24/099 and by CyberSecurity Research Flanders with reference number VOEWICS02.
The second-listed author is additionally supported by the Fonds
voor Wetenschappelijk Onderzoek (FWO) under grant number 1139225N.

\section{Preliminaries}
\label{sec:prelim}
We assume basic familiarity with elliptic curves and their arithmetic, see~\cite{silverman} for more background.
Throughout, $\Fq$ denotes a finite field of characteristic $p$, and $\Fpbar$ its algebraic closure.

\subsection{Supersingular Elliptic Curves}
Supersingularity of elliptic curves over $\mathbb{F}_q$ can be defined by several equivalent properties. Among these, is that their endomorphism ring is isomorphic to a maximal order in a quaternion
algebra, whereas ordinary curves have endomorphism rings isomorphic to quadratic orders.
For this work, essential properties of supersingular elliptic curves are that
their $p$-rank is 0, meaning that $E[p](\Fpbar) = \zero$, and that the trace of the characteristic polynomial of
Frobenius is a multiple of $p$. A shared feature of all these characteristics is that they enforce strict
group structures on $E(\Fq)$, in contrast with ordinary curves, whose group structures are much more flexible.

Supersingular curves are fundamental in isogeny-based cryptography, as their rich endomorphism ring structures allow
for flexibility and efficiency in the design of cryptographic algorithms. Furthermore, the precise control on the group
structure allows for efficient isogeny computations using Velu's formulas~\cite{velu}.

\subsection{Abelian Varieties}
A \emph{polarization} of an abelian variety $A$ is an isogeny $\lambda$ between $A$ and its dual $\hat{A}$
induced by an ample line bundle.
Such a polarization is \emph{principal} whenever $\lambda$ is an isomorphism. In this work we mostly focus on abelian surfaces, which are abelian varieties of dimension 2.
A principally polarized abelian surface over $\Fq$ is either
the Jacobian of a hyperelliptic curve of genus 2, the product of two elliptic curves over $\Fq$,
or the Weil restriction of an elliptic curve over $\mathbb{F}_{q^2}$. All three come equipped with a canonical
principal polarization.

For brevity, we will often use the term abelian surface instead of \emph{principally polarized abelian surfaces}.
We remark, however, that the polarization is used in only two
ways: through the classification above,
and results derived from it
(e.g., \Cref{thm:superspeciality if twist} and \Cref{thm:twisting superspecial});
and through the Weil-pairing isotropy conditions of
\eqref{lem:weil-pairing}, used in the pairing-based conclusive tests in \Cref{sec:conclusive}.

Furthermore, in this work, we mainly focus on Jacobians of hyperelliptic curves,
as supersingularity testing for Weil restrictions and products is easy:
over a quadratic extension, a Weil restriction is isomorphic to a product, and
a product is supersingular, and in particular superspecial, if and only if both factors are supersingular.

\paragraph{Structure.}
The $n$-torsion of an abelian variety over a finite field are those points $P \in A(\Fq)$
such that $[n]P = \zero$. Let $A[n]$ denote the subgroup of $n$-torsion points over the algebraic
closure $\Fpbar$. If $\dim A = g$ and $n$ is coprime to $p$, then
\[
    A[n] \cong \left( \ZZ/n\ZZ \right)^{2g}.
\]
The $\Fq$-rational $n$-torsion $A[n](\Fq)$ is significantly more restricted, due to isotropy conditions enforced by
the Weil pairing, which is where the principal polarization is used as it identifies $A$ with $\hat{A}$
and thereby turns the Weil pairing into a non-degenerate alternating pairing on $A[n]$.
We follow~\cite[Prop. 5.78]{handbook}: for a principally polarized abelian variety $A/\Fq$ of
dimension $g$ and an integer $n$ coprime to $p$,
let $d = \gcd (n, q-1)$. Then we have
\begin{equation}\label{lem:weil-pairing}
    A[n](\Fq) \hookrightarrow (\ZZ_n)^g \times (\ZZ_d)^g.
\end{equation}
In contrast, the $p$-torsion $A[p]$ of an abelian variety
is a more nuanced object,
and in general has rank at most $g$ instead of $2g$, referred to as the \emph{$p$-rank} of $A$.

We refer to the smallest field extension $\Fqk$ such that $A[n] \subseteq A(\Fqk)$
as the \emph{full $n$-torsion degree} of $A$. This is closely related to the \emph{embedding degree}
$k'$ of $n$, defined by the smallest $k'$ such that $n \mid q^{k'} - 1$, i.e., $\mu_n \subseteq \Fqk^\times$.
By the non-degeneracy of the Weil pairing $e_n: A[n] \times A[n] \to \mu_n$, we have $k' \leq k$.

From the above, we can derive the general group structure of $A(\Fq)$ as
\[
    A(\Fq) \cong \ZZ/n_0 \ZZ \times \ldots \times \ZZ/n_{r-1} \ZZ
\]
where $r \leq 2g$ is the \emph{rank} of $A$ and $n_{i+1} \mid n_i$ for all $i$. We refer to $n_0$ as the \emph{exponent}
of $A(\Fq)$, as we have $[n_0] P = \zero$ for all $P \in A(\Fq)$, whereas the order is $|A(\Fq)| = \prod n_i$.

\paragraph{Supersingularity.} In higher dimensions,
supersingularity can be generalized in several ways. A distinguishing property for elliptic curves
is the $p$-rank, which is 1 for ordinary curves, and 0 for supersingular curves.
Similarly, in dimension 2, ordinary abelian surfaces have $p$-rank 2 and supersingular abelian surfaces have $p$-rank 0.
Equivalently, supersingular abelian surfaces are isogenous, as unpolarized varieties, to a product of supersingular elliptic curves.
Abelian surfaces are \emph{superspecial} whenever they are \emph{isomorphic} to a product of supersingular elliptic curves
as unpolarized varieties. Hence, superspecial varieties are always supersingular, but not the other way around.
For Jacobians of hyperelliptic curves, we say that the associated curve is supersingular, or superspecial, if the Jacobian is supersingular or superspecial, respectively.

\subsection{Characteristic Polynomial of Frobenius}
Let $A / \Fq$ be an abelian variety of dimension $g$, then
we denote by $\phi_A$ the characteristic polynomial of the $q$-Frobenius endomorphism $\pi : A \to A$, also called the \emph{Weil polynomial}.
This polynomial captures essential information for abelian varieties:
\begin{enumerate}
    \item We have $\phi_A(x) = \prod_{i = 1}^{2g} (x - \alpha_i)$, where the $\alpha_i$ are the $q$-Weil numbers with $|\alpha_i| = \sqrt{q}$,
    \item We have $\phi_A(1) = |A(\Fq)|$,
    \item Two varieties $A$ and $B$ are isogenous if and only if $\phi_A = \phi_B$,
	\item If $\phi_q(x) = \prod_{i = 1}^{2g} (x - \alpha_i)$ is the Weil polynomial of a variety
    with respect to the $q$-Frobenius, then $\phi_{q^k}(x) = \prod_{i = 1}^{2g} (x - \alpha_i^k)$
    is the Weil polynomial with respect to the $q^k$-Frobenius.
\end{enumerate}
For abelian surfaces over $\Fp$, the characteristic polynomial is of the form
\begin{equation}\label{eq:general-char-pol}
	\phi_A(x) = x^4 + a x^3 - b x^2 + a p x + p^2, \qquad a, b \in \ZZ.
\end{equation}
Since $-a$ is the trace of $\pi$, by an abuse of language we will refer to $a$ as the \emph{trace} of $A$ over $\Fp$.
Most pertinent to this work is the fact that
abelian surfaces are supersingular if and only if they have trace $a = 0$ and $p \mid b$.
For a given abelian surface $A$, there are only two variables, so in general
we can explicitly derive $\phi_A(x)$ if we know the evaluation of $\phi_A$ at two values.

\paragraph{Bounds.}
Not any characteristic polynomial $\phi_A(x)$ is possible over $\Fq$. In general, we get the following
bounds~\cite{weil-bounds}, which we refer to as \emph{Weil bounds}
\begin{align}
    |a| &\leq 4 \sqrt{q}, \label{eq:bound-a} \\
    \phantom{|}b\phantom{|}  &\leq 2q - 2 \sqrt{q} |a|, \label{eq:upper-bound-b} \\
    \phantom{|}b\phantom{|}  &\geq -2q - \frac{1}{4} a^2. \label{eq:lower-bound-b}
\end{align}

\paragraph{The Hasse-Weil Interval.}
One may see from the shape of $\phi_A$ that an abelian variety $A$ of dimension $g$ over $\Fq$
must have at least $(\sqrt{q} - 1)^{2g}$ points, and at most $(\sqrt{q} + 1)^{2g}$ points.
This is the Hasse--Weil interval:
\begin{equation}\label{eq: hasse weil interval}
    (\sqrt{q} - 1)^{2g} \ \leq \ |A(\Fq)| \ \leq \ (\sqrt{q} + 1)^{2g}.
\end{equation}
In general, we find that the number of points is always $O(q^g)$ and the width
of this interval is $O(4g q^{g - 1/2})$. For abelian surfaces over $\Fp$,
this gives roughly $p^2$ points, and a width of $8 \sqrt{p} (p + 1)$ for the interval.
In this work, we will often use the following result, which we refer to as the \emph{width-bound argument}.

\begin{lemma}\label{lem:width-bound-argument}
    Let $A$ be an abelian variety over $\Fq$ of dimension $g$ and $H \subseteq A(\Fq)$
    with $|H|$ larger than the width of the Hasse--Weil interval. Then $|A(\Fq)| = n$,
    where $n$ is the unique multiple of $|H|$ in the Hasse--Weil interval.
\end{lemma}
\begin{proof}
    The generalized width of the Hasse--Weil interval is
    \[
        (\sqrt{q} + 1)^{2g} - (\sqrt{q} - 1)^{2g} \leq 4^g \cdot q^{g - 1/2}
    \]
    and the lower bound of the Hasse--Weil interval is
    $(\sqrt{q} - 1)^{2g}$.
    Thus, for any subgroup $H \subseteq A(\Fq)$ such that $|H|$ is
    larger than the width of the Hasse--Weil interval, there is only one unique multiple $n$ of $|H|$
    in the Hasse--Weil interval. But this must then be the size of $A(\Fq)$, as $|A(\Fq)|$ is in the Hasse--Weil
    interval, and a multiple of $|H|$.
\end{proof}

\paragraph{Extremal \texorpdfstring{Abelian Varieties}{Abelian Varieties}.} An abelian variety that reaches the lower bound is called \emph{minimal}, one that reaches the upper bound
is called \emph{maximal}. We say that an abelian variety is \emph{extremal} if it is either minimal or maximal.
Such extremal varieties are special, as the following result shows.

\begin{lemma}\label{lem:poonen en co}
	If $A$ is an extremal abelian variety over $\Fpt$, then $A$ is superspecial.
\end{lemma}
\begin{proof}
    By \cite[Theorem 5.3]{max_abvars}, a maximal abelian variety $A$ over $\Fpt$
    is isomorphic to a product of maximal, hence supersingular, elliptic curves over $\Fpt$.
    Thus, $A$ is superspecial. The analogous argument holds for minimal abelian varieties.
\end{proof}

\paragraph{Twists.}
An abelian variety that is isomorphic to an abelian surface $A$ over an extension field but not the base field
is called a \emph{twist} of $A$.
Twists are closely related to
automorphisms of the variety~\cite[Ch. X, Thm 2.2]{silverman}. For Jacobians, any twist is induced by a twist of the associated hyperelliptic curve.
Generally, when clear from context, we may refer to
\emph{the} twist $A^t$ when we refer to the quadratic twist induced by $[-1] \in \Aut(A)$.

\section{Supersingular Abelian Surfaces}
\label{sec:structure}
The possible orders and group structures of abelian surfaces over $\Fp$, and the existence
of an abelian surface with a given group structure, are established in several
places throughout the literature. In this section, we list and derive the essential results
on which we base our later supersingularity tests.

Specializing to dimension $g = 2$, our approach is broadly as follows:
we verify supersingularity of an abelian surface by determining the characteristic polynomial $\phi_{A}(x)$.
Although the order and/or group structure of $A(\Fp)$ does not by itself fix the characteristic polynomial,
when combined with the order and/or group structure of either $A(\Fpt)$ or $A^t(\Fp)$,
the characteristic polynomial $\phi_{A}(x)$ is fixed.
We therefore derive the necessary results on the possible shapes
of $\phi_{A}(x)$, the associated orders, and the group structures
of $A$ and its twist.\footnote{
	The notation $A^t$ in this section refers to the
	quadratic twist over the respective field,
	hence may denote different abelian surfaces depending on the base field. Furthermore, $\rho$
	refers to the associated twisting isomorphism.}

\subsection{Group Structures of Abelian Surfaces}\label{subsec:group-structures-surfaces}
We explore the possible group structures of abelian surfaces over $\Fq$ of characteristic $p$.\footnote{
	From here onwards, we assume $p > 5$; there are esoteric abelian surfaces for $p \leq 5$.
}
Surprisingly, they are quite limited, with strong restrictions imposed by both the Weil polynomial
and the Weil bounds.

\paragraph{The Weil Polynomial.} \textcite{rybakov} showed that the Weil polynomial imposes precise
limitations on the possible group structures of principally
polarized abelian varieties over $\Fq$.
We reformulate this result for abelian surfaces.

\begin{lemma}[Cor. 3.1, \cite{rybakov}]\label{lem:probability}
    Let $A$ be an abelian surface over $\Fq$ with characteristic polynomial of Frobenius
    $\phi_A(x) = x^4 + ax^3 - bx^2 + qax + q^2$ with group structure
    \[
            A(\Fq) \cong \ZZ_{n_0} \times \ZZ_{n_1} \times \ZZ_{n_2} \times \ZZ_{n_3},
    \]
    where $n_{i+1} \mid n_i$. Write $\phi_A(1- x)$ as $\lambda_0 + \lambda_1 x + \lambda_2 x^2 + \lambda_3 x^3 + x^4$.
    Then,
    \[
        n_3 \mid \lambda_3, \quad n_2 n_3 \mid \lambda_2, \quad n_1 n_2 n_3 \mid \lambda_1, \quad n_0 n_1 n_2 n_3 \mid \lambda_0.
    \]
\end{lemma}
Rewriting $\phi_A(1-x)$ allows us to express the coefficients $\lambda_i$ in terms of $a$
and $b$.
\begin{corollary}\label{cor:lam-in-terms-phi}
    Over $\Fq$, given $\phi_A(x) = x^4 + ax^3 - bx^2 + qax + q^2$, then
    \begin{align*}
        \lambda_0 &= q^2 + 1 + a(q+1) - b,\\
        \lambda_1 &= -a(q + 1) - 2(a - b + 2),\\
        \lambda_2 &= 3a - b + 6,\\
        \lambda_3 &= -(a + 4).
    \end{align*}
\end{corollary}

\subsection{Group Structures of Supersingular \texorpdfstring{Abelian Surfaces}{Abelian Surfaces}}
We refine our characterization of group structures to the case of supersingular abelian surfaces over $\Fp$.
Whereas supersingular elliptic curves over $\Fp$ have only one possible characteristic polynomial and order, and
two possible group structures,
supersingular abelian surfaces now have five different characteristic polynomials and nine different group structures.
Note that in this section we write $b$ for minus the coefficient of $x^2$ \emph{divided by $p$} in the characteristic polynomial of Frobenius of a supersingular abelian surface.

\begin{lemma}\label{lem:supersingular_orders}
	Let $A$ be an abelian surface over $\Fp$. Then $A$ is supersingular if and only if
	its characteristic polynomial of Frobenius is of the form
	\[
		\phi_A(x) = x^4 - b p x^2 + p^2, \qquad b \in \{-2, -1, 0, 1, 2\}.
	\]
	Consequently, the possible orders of supersingular abelian surfaces over $\Fp$ are
	\[
		|A(\Fp)| \  =  \ \begin{cases}
			\quad (p+1)^2 & \text{if } b = -2			\\
			\quad p^2 + p + 1 & \text{if } b = -1		\\
			\quad p^2 + 1 & \text{if } b = 0			\\
			\quad p^2 - p + 1 & \text{if } b = 1		\\
			\quad (p-1)^2 & \text{if } b = 2
		\end{cases}
	\]
\end{lemma}
\begin{proof}
	The result on $\phi_A(x)$ follows from \cite{howe2009}. The possible orders are then derived from $|A(\Fp)| = \phi_A(1)$.
\end{proof}
These five types of abelian surfaces will be crucial to many of the results in this work, and we therefore introduce the following notation.

\begin{definition}\label{def:supersingular-type}
	Let $A$ be a supersingular abelian surface over $\Fp$.
	When $\phi_A(x) = x^4 - b p x^2 + p^2$, we say that $A$ is \emph{of type $b$}.
	Furthermore, when $n = |A(\Fp)|$, i.e., $n = p^2 - bp + 1$ for $b \in \{-2, -1, 0, 1, 2 \}$,
	then $n$ is a \emph{supersingular order}.
\end{definition}

\noindent Ignoring the 2-torsion, the group structures of these abelian surfaces are simple:
those with parameter $b = \pm 2$ are rank 2, whereas for $b \in \{-1, 0, 1\}$
the groups are cyclic. Including the 2-torsion, we get the following more precise result.

\begin{theorem}[\cite{zhu2000group}]\label{thm: possible group structures of rational jacobians}
	Let $A /\Fp$ be a supersingular abelian surface with $p > 5$.
	\begin{enumerate}
		\item If $b \in \{-1, 0, 1\}$, then the group structure is cyclic and:
		\begin{enumerate}
			\item if $b = -1$, then $A(\Fp) \cong \ZZ/(p^2 + p + 1)\ZZ,$
			\item if $b = \phantom{-}0$, then  $A(\Fp) \cong \ZZ/(p^2 + 1)\ZZ,$
			\item if $b = \phantom{-}1$, then  $A(\Fp) \cong \ZZ/(p^2 - p + 1)\ZZ,$
		\end{enumerate}
		\item if $b = \pm 2$, then the group structure is of rank at least 2 and:
		\begin{enumerate}
			\item if $b = -2$ and $p \equiv 1 \bmod 4$, then $A(\Fp) \cong (\ZZ/(p+1)\ZZ)^2$,
			\item if $b = -2$ and $p \equiv 3 \bmod 4$, then $A(\Fp)$ is isomorphic to \[(\ZZ/(p+1)\ZZ)^2 \quad \text{ or } \quad \ZZ/(p+1)\ZZ \times \ZZ/(\tfrac{p+1}{2})\ZZ \times \ZZ/2\ZZ \quad \text{ or } \quad  (\ZZ/(\tfrac{p+1}{2})\ZZ \times \ZZ/2\ZZ)^2,\]
			\item if $b = 2$ and $p \equiv 1 \bmod 4$, then $A(\Fp)$ is isomorphic to \[(\ZZ/(p-1)\ZZ)^2 \quad \text{ or } \quad \ZZ/(p-1)\ZZ \times \ZZ/(\tfrac{p-1}{2})\ZZ \times \ZZ/2\ZZ \quad \text{ or } \quad  (\ZZ/(\tfrac{p-1}{2})\ZZ \times \ZZ/2\ZZ)^2,\]
			\item if $b = 2$ and $p \equiv 3 \bmod 4$, then $A(\Fp) \cong (\ZZ/(p-1)\ZZ)^2$.
		\end{enumerate}
	\end{enumerate}
\end{theorem}
\begin{remark}
\noindent By \cite{howe2009}, for every prime $p > 5$ there exists a supersingular Jacobian $\J$
with such a group structure, except for the case
$b = 1$ for primes $p \equiv 1 \bmod 6$.
	Note that the reverse of
	\Cref{thm: possible group structures of rational jacobians} is \emph{not} true:
	such group structures exist for non-supersingular Jacobians $\J/\Fp$ too.\footnote{
	For example, for $p = 7$, the Jacobian $\J$ of the hyperelliptic curve defined by
	$y^2=4 x^6+3 x^5+x^4+5 x^3+6 x^2+5$ satisfies $\J(\Fp) \cong \ZZ/(p^2 + 1)\ZZ$.}
\end{remark}
Using \Cref{thm: possible group structures of rational jacobians}, we can derive the group structure up to the rank of the 2-torsion
given only the order of a supersingular Jacobian.
We derive the rank of the 2-torsion $\J[2](\Fp)$ using the factorization of the polynomial defining the associated hyperelliptic curve.\footnote{This rank is important in later sections,
as it decides isotropy behaviour of eigenspaces, see e.g., \Cref{sec:isotropy-conditions-min-maxc,rem:2-rank from f}.}
We say that a non-zero polynomial $f \in \Fp[x]$ has factorization type $\{ d_1, \ldots d_N \}$ over $\Fp$ if it factors as $f = c\prod_{i = 1}^N f_i$ where the factors $f_i$ are distinct, monic, irreducible over $\Fp$, and of degree $d_i$, and $c \in \Fp$ is the leading coefficient.

\begin{lemma}[\cite{two-torsion}]\label{lem: factoring fx gives 2 rank}
	Let $\J$ be the Jacobian over $\Fp$ of a genus-2 hyperelliptic curve $C$ defined by $y^2 = f(x)$, with $\deg f \in \{5, 6\}$.
	Given the factorization $f(x) = c \prod_{i=1}^N f_i(x)$, with $f_i$ distinct monic irreducible polynomials over $\Fp$
	and leading coefficient $c$,
	let $d_i$ denote the degree of $f_i$, so that $\{ d_1, d_2, \ldots d_N \}$ is the factorization type of $f$ over $\Fp$.
	Then $\J[2](\Fp)$ is isomorphic to
	\begin{itemize}
		\item $\langle 1 \rangle$, if $f$ is irreducible, or $\deg(f) = 6$ and $f$ factors as $\{5,1\}, \{3,3\}$;
		\item $(\Z/2\Z)$, if $\deg(f)=6$ and $f$ factors as $\{4, 1, 1\}, \{3, 2, 1 \}, \{4,2\}$, or $\deg(f) = 5$ and $f$ factors as $\{4,1\}, \{3,2\}$;
		\item $(\Z/2\Z)^{2}$, if $\deg(f)=6$ and $f$ factors as $\{3, 1, 1, 1\}, \{2, 2 ,1, 1\}, \{2,2,2\}$, or $\deg(f) = 5$ and $f$ factors as $\{3,1,1\}, \{2,2,1\}$;
		\item $(\Z/2\Z)^{3}$, if $\deg(f)=6$ and $f$ factors as $\{2, 1, 1, 1, 1\}$, or $\deg(f) = 5$ and $f$ factors as $\{2, 1, 1, 1\}$;
		\item $(\Z/2\Z)^{4}$, if $f$ splits into linear factors.
	\end{itemize}
\end{lemma}
Recall that we generally need two evaluations of $\phi_A(x)$ to reconstruct $\phi_A(x)$,
which we use to verify the supersingularity of abelian surface.
While a first evaluation is given by the order $|A(\Fp)| = \phi_A(1)$,
we can get a second evaluation of $\phi_A(x)$ by deriving the order of the twist,
where we use that supersingular abelian varieties have trace $a=0$.

\begin{lemma}\label{lem: the twist is enough}
	An abelian surface $A/\Fp$ has trace $a = 0$ iff $|A(\Fp)| = |A^t(\Fp)|$.
\end{lemma}
\begin{proof}
	We get $\phi_{A^t}(x) = \phi_{A}(-x)$ by \cite[Lem. 3]{sutherland2009}. For
		\[
		\phi_A(x) = x^4 + ax^3 - bx^2 + apx + p^2
	\]
		the equality of $\phi_A(1) = |A(\Fp)|$ and $\phi_{A^t}(1) = |A^t(\Fp)|$ implies
	$a + ap = -a - ap$, thus $a = 0$. Conversely, if $a = 0$, we get $\phi_A(x) = \phi_A(-x) = \phi_{A^t}(x)$.
\end{proof}

\begin{remark}\label{rem:twist-iso}
	As $A[2](\Fp) = A^t[2](\Fp)$, we may combine the above result with \Cref{thm: possible group structures of rational jacobians}
	to derive the stronger result that for supersingular $A$,
	as $a = 0$, we get $A(\Fp) \cong A^t(\Fp)$.
\end{remark}

\subsection{Uniqueness of Supersingular Orders}\label{sec:unique-orders}
From \Cref{lem:supersingular_orders}, we know that if $A/\Fp$ is supersingular,
then its order $n$ is of the form $n = p^2 - bp + 1$ for $b \in \{-2, -1, 0, 1, 2\}$.
The converse is (almost) true, except we have to look over $\Fpt$:
if $A/\Fp$ is an abelian surface with $|A(\Fpt)| = n^2$ for a supersingular order $n$, then $A$ is supersingular.
The goal of this section will be proving this result, which is a key ingredient in our supersingularity testing.

At a high level, our argument is as follows.
We first observe that even though for a supersingular order $n$ the equality $|A(\Fp)| = n$ can be satisfied by both
supersingular and non-supersingular abelian surfaces,
for the cases $b \in \{ \pm 1, \pm 2\}$
the equality $|A(\Fpt)| = n^2$ is only satisfied by supersingular abelian surfaces.
The case $b = 0$ is more nuanced: there exist non-supersingular abelian surfaces $A$
with $|A(\Fpt)| = (p^2 + 1)^2$, but such abelian surfaces are never defined over $\Fp$.
More precisely, we show that any abelian surface $A / \Fp$ with $|A(\Fpt)| = n^2$ is supersingular.

In this section, we often require the Weil polynomial of an abelian surface $A$ over
both $\Fp$ and $\Fpt$.
When the abelian surface $A$ is clear from context, we write the Weil polynomial over $\Fp$ as
\[
	\phi_p(x) = x^4 + a_p x^3 - b_p x^2 + a_p p x + p^2
\]
and the Weil polynomial over $\Fpt$, using $q = p^2$, as
\[
	\phi_q(x) = x^4 + a_q x^3 - b_q x^2 + a_q p^2 x + p^4.
\]
Note that $a_p$, $b_p$ and $a_q$, $b_q$ denote the (unnormalised) coefficients of $\phi_p$ and $\phi_q$ as displayed, so that $A$ is of type $b = b_p/p$ in the sense of \Cref{def:supersingular-type}.
In the proofs of \Cref{lem:order-p2-p-1} and \Cref{lemma:order-p2-1}, it is convenient to work with the
$x^2$-coefficients $c_p \colonequals -b_p$ and $c_q \colonequals -b_q$.

\paragraph{The case $\mathbf{b = \pm 2}$.}
\Cref{lem:poonen en co} resolves the case $b = \pm 2$: any abelian surface $A$
with $|A(\Fpt)| = (p \pm 1)^4$ is either minimal or maximal, and is thus supersingular
(and even superspecial).

\paragraph{The case $\mathbf{b = \pm 1}$.}
We derive the result
from the central requirement that $\phi_A(1) = n^2$ over $\Fpt$, with $n = p^2 \pm p + 1$, together with the Weil bounds.

\begin{lemma}\label{lem:order-p2-p-1}
	Let $n = p^2 \pm p + 1$ and let $A$ be an abelian surface such that $|A(\Fpt)| = n^2$.
	Then the Weil polynomial
	of $A$ over $\Fpt$ is \[(x^2 \pm px +p^2)^2 = x^4 \pm 2px^3 + 3p^2x^2 \pm 2p^3x + p^4,\] and in particular $A$
	is supersingular. Moreover, if $A$ is defined over $\Fp$, then $A$ is of type $b = \mp 1$.
\end{lemma}

\begin{proof}
	We derive the result for $n = p^2 + p + 1$, which requires us to show
	that $a_q = 2p$ and $c_q = 3p^2$.
	The analogous result can be derived for $n = p^2 - p + 1$ simply by a change of sign in the right places.
	With $q = p^2$, the Weil bounds in \Cref{eq:bound-a,eq:lower-bound-b,eq:upper-bound-b}
	become
	\[
		|a_q| \leq 4p, \quad
		c_q \leq 2p^2 + \frac{1}{4} a_q^2, \quad \text{and} \quad
		c_q \geq -2p^2 + 2p \cdot |a_q|.
	\]
	As $|A(\Fpt)| = n^2 = \phi_A(1)$, we get
	\begin{equation}\label{eq:p2-p-1}
		a_q(p^2 + 1) + c_q = 2p^3 + 3p^2 + 2p.
	\end{equation}
	We now derive the two bounds $a_q \geq 2p$ and $a_q \leq 2p$
	to conclude that $a_q = 2p$, which directly implies $c_q = 3p^2$.
	First, by $c_q \leq 2p^2 + \frac{1}{4} a_q^2$, and substituting $c_q$
	into \Cref{eq:p2-p-1}, we get
	\[
		\frac{1}{4} a_q^2 + a_q(p^2 + 1) - (2p^3 + p^2 + 2p) \geq 0,
	\]
	where equality with 0 is reached at $a_q = 2p$ and $a_q = -4p^2 - 2p - 4$.
	The bound $|a_q| \leq 4p$ excludes the second branch, so we get
	$2p \leq a_q \leq 4p$.
	Then, by $c_q \geq -2p^2 + 2p |a_q|$ and substituting $c_q$, noticing $|a_q| = a_q$ as $a_q \geq 2p$, we get
	\[
		2p a_q - 2p^2 \leq 2p^3 + 3p^2 + 2p - a_q(p^2 + 1),
	\]
	which can be rewritten to
	\[
		a_q \leq 2p + \frac{p^2}{(p+1)^2}.
	\]
	This last fraction is always smaller than 1, so as $a_q$ is an integer we have that $a_q \leq 2p$.
	We conclude that $a_q = 2p$ and $c_q = 3p^2$, i.e., $b_q = -3p^2$, as required.

	Now suppose that $A$ is the base change of an abelian surface over $\Fp$, with Weil polynomial $\phi_p$.
	We have that
	\begin{align}\label{eq:relation-aq-cq}
		a_q &= 2c_p - a_p^2, \quad c_q = c_p^2 - 2 a_p^2 p + 2p^2.
	\end{align}
	Substituting into $a_q = 2p$ and $c_q = 3p^2$, we get $a_p^2(a_p^2 - 4p) = 0$. As $p>5$, $4p$ is not a square, so $a_p = 0$ and $c_p = p$ and $A$ is of type $b = -1$.
\end{proof}

\paragraph{The case $\mathbf{b = 0}$.}
This case is most interesting, as for $n = p^2 + 1$
there exist non-supersingular abelian surfaces such that
$|A(\Fpt)| = n^2$, in contrast with the other cases. However, we show
that such abelian surfaces are never defined over $\Fp$, except for the supersingular ones.

\begin{lemma}\label{lemma:order-p2-1}
	Let $n = p^2 + 1$ and let $A$ be an abelian surface over $\Fp$ with $|A(\Fpt)| = n^2$.
	Then $A$ is supersingular, and in particular, the Weil polynomial
	of $A$ over $\Fp$ is $x^4 + p^2$.
\end{lemma}
\begin{proof}
	We denote the Weil polynomial for $A$ over $\Fp$
	by $\phi_p(x)$
	and over $\Fpt$ by $\phi_q(x)$. The goal is to use  \Cref{lem:supersingular_orders} and show that $a_p =0$ and $p \mid c_p$.

	As $|A(\Fpt)| = n^2$ we get from $\phi_q(1) = n^2$ that
	\begin{equation}\label{eq:A}
		a_q(p^2 + 1) + c_q = 2p^2.
	\end{equation}

	By $\phi_p(x) = \prod (x - \alpha_i)$ and $\phi_q(x) = \prod (x - \alpha_i^2)$,
	we have that $\varphi_p(x)\varphi_p(-x) = \varphi_q(x^2)$.
	Expanding this out and comparing coefficients, we
	find that $a_q = 2c_p - a_p^2$ and $c_q = c_p^2 - 2 a_p^2 p + 2p^2$.
	This allows us to rewrite \Cref{eq:A} as
	\begin{equation}\label{eq:B}
		a_p^2 = \frac{c_p^2 + 2c_p(p^2 + 1)}{(p+1)^2},
	\end{equation}
	and so from $a_q = 2c_p - a_p^2$, by expressing everything in terms of $c_p$,

	we get
	\begin{equation}\label{eq:C}
		a_q = \frac{c_p(4p - c_p)}{(p+1)^2}.
	\end{equation}
	First we bound $c_p$. From the Weil bound in \Cref{eq:upper-bound-b} we have
	$c_p \leq 2p + \frac{1}{4} a_p^2$.
	Together with \Cref{eq:B}, we obtain
	\[ c^2_p - 2c_p(p^2+4p+1) + 8p(p+1)^2 \geq 0.\]
	Solving the quadratic in $c_p$, we get that $c_p \leq 4p$ or $c_p \geq 2(p^2 + 2p + 1)$. The latter is not possible due to the Weil bounds, so we find that
	$c_p \leq 4p$.

	Furthermore, the right-hand side of \Cref{eq:B} should be positive,
	which implies that $c_p^2+2c_p(p^2+1) \geq 0$.
	Writing as $c_p(c_p+2p^2+2) \geq 0$, and noting that
	$c_p + 2p^2 + 2 \geq 0$ by \Cref{eq:lower-bound-b}, we have that $c_p \geq 0$.
	Thus, we know $0 \leq c_p \leq 4p$.
	Using these values for $c_p$ in \Cref{eq:C}, and noting that $a_q$ is an integer, we have $a_q \in \{0, 1, 2, 3\}$. We now show that only $a_q = 0$ is possible.
	\begin{itemize}
		\item[$a_q = 1$:] From $(p+1)^2 = c_p (4p - c_p)$, we get $2 \mid c_p$.
		But then $1 = a_q = 2c_p - a_p^2$ cannot have a solution, as this requires $a_p$ odd
		and $a_p^2 \equiv 3 \bmod 4$, which is impossible.
		\item[$a_q = 2$:] From $2(p+1)^2 = c_p (4p - c_p)$, we get $4 \mid c_p$.
		But then $2 = a_q = 2c_p - a_p^2$ cannot have a solution mod $8$.
		\item[$a_q = 3$:] To get an integer solution to the quadratic in $c_p$ from \Cref{eq:C}, the discriminant needs to be a square. This discriminant is
		equal to $(p-3)^2 - 12$ and hence only a square for $p = 7$.
		Thus, for $p > 7$, this case is never possible. For $p = 7$, the solutions
		$c_p = 12, 16$ force a non-integer solution for $a_p$ in $3 = a_q = 2c_p - a_p^2$.
	\end{itemize}
	We conclude that $a_q = 0$ and hence $c_q = 2p^2$ by \Cref{eq:A}.
	This is enough to show supersingularity of $A$, but we can even deduce from \Cref{eq:C} that
	$c_p = 0$, so that $a_p = 0$ (from $a_q = 2c_p - a_p^2$), as the case $c_p = 4p$
	forces $a_p^2 = 8p$, which can only happen for $p = 2$.
\end{proof}
We summarize the results of this section in the following theorem.
\begin{theorem}\label{thm:supersingular-orders}
	Let $n = p^2 - bp + 1$ with $b \in \{-2, -1, 0, 1, 2\}$.
	\begin{itemize}
		\item Assuming $b \neq 0$, if $A$ is an abelian surface over $\Fp$ or $\Fpt$
		such that $|A(\Fpt)| = n^2$,
		then $A$ is supersingular.
		\item If $b = 0$ and $A$ is an abelian surface over $\Fp$ such that $|A(\Fpt)| = n^2$,
		then $A$ is supersingular.
	\end{itemize}
\end{theorem}

\subsection{Structure theory of $m$-torsion}\label{sec:pairings-jacobians}
Our conclusive supersingularity testing in \Cref{sec:conclusive} uses the $m$-Weil pairing to identify the structure of the rational
$m$-torsion, where $m$ is the exponent of $A(\Fp)$. The goal of this section is to show that knowing the structure of $A[m]$ is
sufficient to determine the supersingularity of $A$. To this end, we develop theory on the $m$-torsion of supersingular Jacobians,
and derive the necessary results on the Weil pairing to enable this testing.

To determine the full $m$-torsion degree of $A[m]$, we use the following lemma.
\begin{lemma}\label{lem:phi splits mod n}
    Let $\phi(x) = x^4 - bpx^2 + p^2$ and let $m$ be any divisor of $p^2 - bp + 1$. Then
    	\[
        \phi(x) \equiv (x^2 - 1) \cdot (x^2 - p^2) \quad \bmod m.             \eqno
    \]
	\end{lemma}
\begin{remark}\label{rem:torsion-in-Fpt}
	The minimal polynomial of the Frobenius endomorphism $\pi$ is squarefree (as $\pi$ is central in $\End^0(A)$, which is semisimple) and divides $\phi$. In particular, for $A$ of type $b = \pm 2$ we have $\pi^2 = [\pm p]$ on $A$. Hence, $\pi^2$ acts trivially on $A[p\mp 1]$, so $A[p\mp 1] \subseteq A(\Fpt)$.
\end{remark}

As $\phi(x)$ splits completely modulo $m$, if the roots are distinct modulo $\ell$ for all primes $\ell \mid m$, then
the action of Frobenius on $A[m]$ is diagonalizable,
and its order
will be the lowest common multiple of the orders of the four eigenvalues $\pm 1, \pm p \in (\Z/m\Z)^\times$. Hence, we can derive the
full $m$-torsion degree from the embedding degree of $m$, and decompose $A[m]$ accordingly.
Generically, we get $X^2 + X + 1 \mid X^3 - 1$, $X^2 + 1 \mid X^4 - 1$
and $X^2 - X + 1 \mid X^6 - 1$, which gives embedding degrees 3, 4, and 6 for
$m$-th root of unities where $m$ is the exponent for $b$-type $-1$, $0$, and $1$ respectively.
With a slight abuse of notation, for $\lambda \in \{\pm 1, \pm p\}$ we write $A[\pi - \lambda] = \ker(\pi - \lambda) \cap A[m]$
for the corresponding eigenspace. These are defined over the following fields.

\begin{lemma}\label{lem:eigenspace fields}
  Let $A$ over $\Fp$ be a supersingular abelian surface of type $b$ and exponent $m$.
  If $\lambda^k \equiv 1 \bmod m$,
  then $A[\pi - \lambda] \subseteq A(\mathbb{F}_{p^k})$, and if $\lambda^k \equiv -1 \bmod m$, then
  $A[\pi - \lambda] \subseteq A^t(\mathbb{F}_{p^k})^\rho$.
	Furthermore:
  \begin{enumerate}
	\item The eigenspaces of $\lambda = \pm 1$ are
	$A[\pi - 1] = A(\Fp)$ and $A[\pi + 1] = A^t(\Fp)^\rho$.
	\item The eigenspaces of $\pm p$ are, per type $b$:
	\begin{enumerate}
		\item $b = \phantom{-}2$: $A[\pi - p] = A[\pi - 1] = A(\Fp)$ and $A[\pi + p] = A[\pi + 1] = A^t(\Fp)^\rho$.
		\item $b = -2$: $A[\pi - p] = A[\pi + 1] = A^t(\Fp)^\rho$ and $A[\pi + p] = A[\pi - 1] = A(\Fp)$.
		\item $b = \phantom{-}0$: $A[\pi - p] = A^t(\Fpt)^{\rho}$ and $A[\pi + p] = A^t(\Fpt)^{\rho'}$.
		\item $b = -1$: $A[\pi - p] = A(\Fpthree)$ and $A[\pi + p] = A^t(\Fpthree)^{\rho'}$.
		\item $b = \phantom{-}1$: $A[\pi - p] = A^t(\Fpthree)^{\rho}$ and $A[\pi + p] = A(\Fpthree)$.
	\end{enumerate}
\end{enumerate}
\end{lemma}
\begin{proof}
  For $P \in A[\pi - \lambda]$ we have $\pi^k P = [\lambda^k] P$. If $\lambda^k \equiv 1 \bmod m$ this
  equals $P$, so $P$ is fixed by the $p^k$-power Frobenius and hence $P \in A(\mathbb{F}_{p^k})$.
  If $\lambda^k \equiv -1 \bmod m$, we have $\pi^k P = -P$, so $P \in A^t(\mathbb{F}_{p^k})$.

 The congruences follow from $m \mid p^2 - bp + 1$. For $b = \pm 2$ we have
  $m \in \{ p \mp 1, \tfrac{p \mp 1}{2}\}$ by \Cref{thm: possible group structures of rational jacobians},
  so $p \equiv \pm 1 \bmod m$. For $b = 0$ we have $m = p^2 + 1$, so $p^2 \equiv -1 \bmod m$. Finally, for
  $b = \pm 1$ we have $m = p^2 \mp p + 1$ dividing $p^3 \pm 1$, hence $p^3 \equiv \mp 1 \bmod m$.
  The statements for $A[\pi + p]$ follow from $(-p)^k = (-1)^k p^k$.
\end{proof}
The subgroup generated by these four eigenspaces is almost all of $A[m]$; the index is a power of $2$ or
of $3$, and is determined by the type $b$ and the group structure of $A(\Fp)$.

\begin{lemma}\label{lem:decomposition of Jn}
  Let $A$ be a supersingular abelian surface defined over $\Fp$, of type $b$ and exponent $m$. Let
    \[
    H = A[\pi - 1] + A[\pi + 1] +  A[\pi - p] +  A[\pi + p] \subseteq A[m]
  \]
    be the subgroup generated by the four eigenspaces of \Cref{lem:eigenspace fields}. Then the index of
  $H$ in $A[m] \cong (\ZZ/m\ZZ)^4$ is as follows:
  \begin{itemize}
    \item If $b = \pm 2$, then $H = A(\Fp) + A^t(\Fp)^\rho$ and
      \[
        [A[m] : H] \;=\; \frac{m^4 \cdot |A[2](\Fp)|}{|A(\Fp)|^2} \;=\;
        \begin{cases}
          4 & \text{if } A(\Fp) \cong \left( \ZZ/m\ZZ \right)^2, \\
          8 & \text{if } A(\Fp) \cong \ZZ/m\ZZ \times \ZZ/\tfrac{m}{2}\ZZ \times \ZZ/2\ZZ, \\
          1 & \text{if } A(\Fp) \cong \left( \ZZ/m\ZZ \times \ZZ/2\ZZ \right)^2,
        \end{cases}
      \]
      where the first two cases have $m = p \mp 1$ and the third has $m = \tfrac{p \mp 1}{2}$.
    \item If $b = \phantom{-}0$, then $[A[m] : H] = 8$.
    \item If $b = -1$, then $[A[m] : H] = 9$ if $p \equiv 1 \bmod 3$, and $[A[m] : H] = 1$ otherwise.
    \item If $b = \phantom{-}1$, then $[A[m] : H] = 9$ if $p \equiv 2 \bmod 3$, and $[A[m] : H] = 1$ otherwise.
  \end{itemize}
\end{lemma}
\begin{proof}
  Since $A[m] = \bigoplus_{\ell^e \mid m} A[\ell^e]$ and every eigenspace decomposes accordingly, the index
  of $H$ is the product of the local indices $[A[\ell^e] : H \cap A[\ell^e]]$ over the prime powers
  $\ell^e \mid m$. For a given prime factor $\ell$, suppose that $1, -1, p, -p$ are pairwise distinct modulo $\ell$.
  By \Cref{lem:phi splits mod n} we have
  $\phi(x) \equiv \prod_\lambda (x - \lambda) \bmod \ell^e$, and these factors are pairwise coprime
  in $(\ZZ/\ell^e\ZZ)[x]$ because the differences $\lambda - \lambda'$ are units.
  Thus, $A[\ell^e] = \bigoplus_\lambda \ker(\pi - \lambda)$, where each summand is
  free of rank 1 over $\ZZ/\ell^e\ZZ$. This implies that the local index at $\ell$ is 1.
  The only cases where this fails, i.e., $1, -1, p, -p$ are \emph{not} pairwise distinct modulo $\ell$, are:
  $b = 0$ and $\ell = 2$; $b = \pm 1$ and $\ell = 3$; or $b = \pm 2$, where the eigenvalues $\pm 1$ and $\pm p$
  agree (in some order) modulo every $\ell \mid m$.

  Consider first the case $b = \pm 2$.
  By \Cref{lem:eigenspace fields} we have $p \equiv \pm 1 \bmod m$, so $A[\pi - p]$ and $A[\pi + p]$
  coincide, in some order, with $A[\pi - 1] = A(\Fp)$ and $A[\pi + 1] = A^t(\Fp)^\rho$; thus
  $H = A(\Fp) + A^t(\Fp)^\rho$. As $A(\Fp) \cap A^t(\Fp)^\rho = A[2](\Fp)$, we have
  \[
	|H| = \frac{|A(\Fp)| \cdot |A^t(\Fp)|}{|A[2](\Fp)|} = \frac{|A(\Fp)|^2}{|A[2](\Fp)|}.
  \]
  Substituting the three
  group structures of \Cref{thm: possible group structures of rational jacobians} yields the three indices.

  Consider the case $b = 0$, where $m = p^2 + 1$ and $A(\Fp)$ is cyclic by
  \Cref{thm: possible group structures of rational jacobians}. As $p^2 \equiv 1 \bmod 8$
	we get $m \equiv 2 \bmod 8$, so $2 \mid m$ and the $2$-Sylow subgroup of $A(\Fp)$ is $\ZZ/2\ZZ$. All
  four eigenvalues are congruent modulo $2$ to $1$, so each eigenspace has its $2$-part contained in
  $\ker(\pi - 1) \cap A[2] = A[2](\Fp) \cong \ZZ/2\ZZ$. So, $H \cap A[2] \cong \ZZ/2\ZZ$ with local
  index at $2$ being $2^4 / 2 = 8$; other local indices are 1.

  In the case $b = \pm 1$, we have that $m = p^2 \mp p + 1$ is odd, and $3 \mid m$ if and only if $p \equiv \mp 1 \bmod 3$.
  If $3 \nmid m$, the index is 1. If $3 \mid m$, then
  $\phi(x) \equiv (x-1)^2 (x+1)^2 \bmod 3$. Since
  $A(\Fp)$ is cyclic, we get $\ker(\pi - 1) \cap A[3] = A[3](\Fp) \cong \ZZ/3\ZZ$, and
  likewise $\ker(\pi + 1) \cap A[3] \cong A^t[3](\Fp) \cong \ZZ/3\ZZ$.
	Therefore, $H \cap A[3]$ has order $9$ and local index at $3$ is $3^4 / 3^2 = 9$.
\end{proof}

Such a basis of eigenpoints as a subgroup of $A[m]$ is almost symplectic
as the following result on their pairing values shows.

\begin{lemma}\label{lem:almost symplectic basis}
  Let $A$ over $\Fp$ be a supersingular abelian surface of exponent $m$,
  and  let $P \in A[\pi -1], Q \in A[\pi + 1], R \in A[\pi - p], S \in A[\pi + p]$ be eigenpoints in $A[m]$.
  Then $e_m(P, S)$ and $e_m(Q, R)$ are of order at most $\gcd(m, 2)$, and
  $e_m(P, Q)$ and $e_m(R, S)$ are of order at most $\gcd(m, p+1)$.
\end{lemma}
\begin{proof}
  Generically, for two eigenpoints $D$ and $D'$ of eigenvalue $\lambda, \lambda'$ respectively, we may set
  $\zeta = e_m(D, D')$, and derive constraints on $\zeta$ using
    \[
       e_m( \pi(D), \pi(D') ) = \zeta^p \text{ and } e_m( \pi(D), \pi(D') ) = e_m( \lambda D, \lambda' D') = \zeta^{\lambda \cdot \lambda'}.
  \]
    For example, this gives as constraints on $\zeta_{P, Q} = e_m(P, Q)$ that $\zeta_{P, Q}^p = \zeta_{P, Q}^{-1}$ and hence
  $\zeta_{P, Q}^{p+1} = 1$, which implies $\zeta_{P, Q} \in \mu_{p+1}$.
  Similarly, we get
  \[
      \zeta_{P, R}^{0} = 1, \quad \zeta_{P, S}^{2p} = 1, \quad \zeta_{Q, R}^{2p} = 1, \quad \zeta_{Q, S}^{0} = 1,
      \quad \text{ and } \quad \zeta_{R, S}^{p^2 + p} = 1.
  \]
  Since $m$ and $p$ are coprime, this implies $\zeta_{P, S} \in \mu_2$ and
  $\zeta_{Q, R} \in \mu_2$. For $\zeta_{R, S}$, we can remove a factor $p$
  to obtain $\zeta_{R, S}^{p+1} = 1$.
  Thus, $e_m(P, S), e_m(Q, R) \in \mu_{\gcd({m, 2})}$, and $e_m(P, Q), e_m(R, S) \in \mu_{\gcd({m, p+1})}$.
\end{proof}
For $b = \pm 2$, \Cref{lem:almost symplectic basis} leaves an ambiguity at $2$, as it only determines the pairings between the eigenspaces of $\pm 1$ (for $b = 2$), or within them (for $b = -2$), up to $\mu_2$. We resolve it using the trace map $\pi + 1$, which interacts with the Weil pairing as follows. Note that we will work with $n = p \mp 1$ rather than with the exponent $m$ of $A(\Fp)$, so that statements can be made uniformly over the three group structures (since $|A(\Fp)| = n^2$ in all cases).

\begin{lemma}\label{lem:trace-pairing}
Let $A/\Fp$ be a supersingular abelian surface of type $b = \pm 2$, and let
$n = p \mp 1$. Then $A(\Fp) \subseteq A[n] \subseteq A(\Fpt)$, and
\[
  (\pi + 1)\,A[n] = A(\Fp), \qquad (\pi - 1)\,A[n] = A^t(\Fp)^\rho .
\]
Moreover, for all $R \in A[n]$, $Q \in A(\Fp)$ and $S \in A^t(\Fp)^\rho$,
\[
  e_n\big((\pi + 1)R, Q\big) = e_n(R, Q)^{p+1},
  \qquad
  e_n\big((\pi - 1)R, S\big) = e_n(R, S)^{-(p+1)} .
\]
\end{lemma}
\begin{proof}
	The exponent of $A(\Fp)$ divides $n$, so
	$A(\Fp) \subseteq A[n]$; the same holds for $A^t$, which is again of type $b = \pm 2$. By~\Cref{rem:torsion-in-Fpt},
	$A[n] \subseteq A(\Fpt)$ and $(\pi - 1)(\pi + 1) = \pi^2 - 1 = 0$ on $A[n]$.
	On $A[n]$, the kernel of $\pi - 1$ is $A(\Fp)$. The kernel of $\pi + 1$ is $A^t(\Fp)^\rho$ that is of order $n^2$, so $(\pi + 1)A[n]$ has order $n^4/n^2 = n^2$. It is contained in
	$\ker(\pi - 1) = A(\Fp)$, thus $(\pi + 1)A[n] = A(\Fp)$. Similarly,
	$(\pi - 1)A[n] = A^t(\Fp)^\rho$.

	Moreover,
	$e_n(\pi R, \pi T) = e_n(R, T)^p$ for all $R, T \in A[n]$. If $Q \in A(\Fp)$, then $e_n(\pi R, Q) = e_n(\pi R, \pi Q) = e_n(R, Q)^p$, and multiplying by
	$e_n(R, Q)$ gives the first identity. For $S \in A^t(\Fp)^\rho$, we have $\pi S = -S$ giving the second identity.
\end{proof}
An overview of these structural results is given in \Cref{tab:jacobian-summary}.

\begin{table}[ht!]
	\centering
	\caption{Properties of supersingular abelian surfaces $A/\Fp$ by $b$-type, for $p > 5$.
	  Here $m$ denotes the exponent of $A(\Fp)$, the embedding degree is the smallest $k$ with $m \mid p^k - 1$,
	  and the full $m$-torsion field is the smallest extension of $\Fp$ over which $A[m]$ is defined.}
	\label{tab:jacobian-summary}
	\vspace{6pt}
	\renewcommand{\arraystretch}{1.4}
	\begin{tabular*}{\textwidth}{@{\extracolsep{\fill} } c c c c c}
	\toprule
	$b$ & Group structure $A(\Fp)$ & $|A(\Fp)|$ & Emb.\ deg.\ of $m$ & Full $m$-torsion \\
	\midrule
	$-2$ &
	  \makecell[c]{
				$(\ZZ/(p+1)\ZZ)^2$ or\\[2pt]
				$\ZZ/(p+1)\ZZ \times \ZZ/\tfrac{p+1}{2}\ZZ \times \ZZ/2\ZZ$ or\\[2pt]
				$(\ZZ/\tfrac{p+1}{2}\ZZ \times \ZZ/2\ZZ)^2$
				}
	  & $(p+1)^2$ & $2$ & $\Fpt$ \\[6pt]
	\midrule
	$-1$ & $\ZZ/(p^2+p+1)\ZZ$ & $p^2+p+1$ & $3$ & $\Fps$ \\
	\midrule
	$0$ & $\ZZ/(p^2+1)\ZZ$  & $p^2+1$ & $4$ & $\Fpf$ \\
	\midrule
	$1$ & $\ZZ/(p^2-p+1)\ZZ$ & $p^2-p+1$ & $6$ & $\Fps$ \\
	\midrule
	$2$ &
	  \makecell[c]{$(\ZZ/(p-1)\ZZ)^2$ or \\[2pt]
				   $\ZZ/(p-1)\ZZ \times \ZZ/\tfrac{p-1}{2}\ZZ \times \ZZ/2\ZZ$ or\\[2pt]
				   $(\ZZ/\tfrac{p-1}{2}\ZZ \times \ZZ/2\ZZ)^2$}
	  & $(p-1)^2$ & $1$ & $\Fpt$ \\
	\bottomrule
	\end{tabular*}
	\medskip

\end{table}

\begin{remark}
	For $b = 1$, a Jacobian with the corresponding group structure does not exist when $p \equiv 1 \pmod{6}$,
	as shown by~\textcite{howe2009}.
	Furthermore,
	when $p \equiv 1 \mod 4$, all three rows for $b = -2$ are equivalent, as $\tfrac{p+1}{2}$ is odd,
	and when $p \equiv 3 \mod 4$, all three rows for $b = 2$ are equivalent, as $\tfrac{p-1}{2}$ is odd.
\end{remark}

\subsection{Superspecial Jacobians}\label{sec:superspecial criterion}

Thus far, we have only considered the supersingular case; it is also important to understand the superspecial case, particularly for cryptographic applications. In this section, we focus in particular on Jacobians, as we will appeal to results about the associated genus-$2$ hyperelliptic curve,
and products of supersingular elliptic curves are by definition superspecial.
To show that a Jacobian is superspecial, we require two results from the literature that allow us to
reduce to determining a simpler property: minimality or maximality. Essentially, superspecial
Jacobians are always related to the supersingular varieties that reach the minimal or maximal
number of points. The first result we need is \Cref{lem:poonen en co}, which relates maximality to superspeciality
for abelian varieties. The next result connects this to curves.

\begin{lemma}[Thm. 2.6, \ \cite{superspecial_note}]\label{lem:superspecial note}
	Let $C$ be a curve over $\Fpt$. Then $C$ is superspecial if and only if $C$ has a twist
	over $\Fpt$ that is minimal or maximal.
\end{lemma}
\begin{remark}
    There is a significant difference in behaviour between Jacobians over $\Fp$ of order $(p+1)^2$,
    and Jacobians over $\Fpt$ of order $(p+1)^4$. Namely, the first category contains both ordinary and supersingular Jacobians,
    and only the supersingular Jacobians of order $(p+1)^2$ become maximal, i.e., of order $(p+1)^4$ over $\Fpt$.
    The second category only contains supersingular, and by \Cref{lem:poonen en co} superspecial, Jacobians.
\end{remark}

From this, we get a simple criterion for superspeciality of principally polarized Jacobians: a Jacobian $\J$ is superspecial if and only if the associated curve $C$ is superspecial. \Cref{lem:superspecial note} shows that
this implies a maximal or minimal twist.
We derive our main criterion for superspeciality as a direct corollary to \Cref{lem:superspecial note}.

\begin{corollary}\label{thm:superspeciality if twist}
	Let $A$ be a principally polarized abelian surface over $\Fp$ or $\Fpt$.
	Then $A$ is superspecial if and only if $A$ has a twist over $\Fpt$ that is minimal or maximal.
\end{corollary}

\noindent We refer to non-minimal or non-maximal superspecial Jacobians as \emph{twisted superspecial Jacobians}
to emphasize the dichotomy between minimal and maximal superspecial Jacobians and those Jacobians that require a twist
to become minimal or maximal.
For the quadratic twist $\J^t$, we know that $\phi_{\J^t}(x) = \phi_\J(-x)$, and this equals $\phi_\J(x)$
for superspecial Jacobians over $\Fp$, as in the proof of \Cref{lem: the twist is enough}.
Hence, twisted superspecial Jacobians must arise from non-trivial twists
over a larger extension field. Using two known results, we identify curves with
additional twists and the effect on characteristic polynomials.

\begin{lemma}[Ch. X, Thm 2.2(c), \cite{silverman}]\label{lem:twists}
	Let $C$ be a curve over a field $k$ and let $G_k$ denote the Galois group $\Gal(\kbar / k)$.
	Then there is a bijection
	\[
		\Twists(C/k) \to H^1(G_k, \Aut(C)).
	\]
\end{lemma}

We can therefore associate to any twist $C'$ of $C$ some automorphism $\sigma \in \Aut(C)$,
and we will denote $C'$ by $C^\sigma$ in such a case.
Luckily, for a hyperelliptic curve $C$ of genus 2, there are not too many options for $\Aut(C)$:
We always have $[-1] \in \Aut(C)$, so we instead analyse
the reduced automorphism group
$\RA(C) := \Aut(C) / \langle -1 \rangle$. \textcite{bolza} shows that there are only 7 options for $\RA(C)$.
For every option of $\RA(C)$, \textcite{cardano} computed the corresponding cohomology group.
Recall that we are mostly interested in the change of the characteristic polynomial $\phi_{C^\sigma}(x)$
in terms of $\phi_C(x)$, which is exactly what \textcite{howe2009} have described for
minimal and maximal curves over $\Fpt$.

\begin{lemma}[Prop. 13.1, \cite{howe2009}]\label{lem:effect-twist-on-minpol}
	Let $C$ be a minimal or maximal supersingular curve over $\Fpt$, so that $\pi$ acts as $[\pm p]$.
	Let $C^\sigma$ be a twist of $C$ induced by $\sigma \in \Aut(C)$, and let $f_\sigma$ denote the characteristic polynomial of $\sigma$.
	Then
		\[
		\phi_{C^\sigma}(x) = p^4 \cdot f_\sigma( x / \pm p).
	\]
	\end{lemma}

\noindent Using \cite[Table 4]{howe2009}, we know all possible minimal polynomials of automorphisms of genus-2 curves.
and may compute which possible automorphisms twist a curve of type $b = \pm 2$ into a curve of type $b \in \{-1, 0, 1\}$.
We describe the non-trivial automorphisms $\tau$ and $\rho$ on genus-2 hyperelliptic curves in \Cref{app:automorphisms}.
\begin{proposition}
	Let $C$ be a hyperelliptic curve of genus 2 over $\Fp$ of type $b = \pm 2$
	 so that $C/\Fpt$ is extremal.
	Then
	\begin{itemize}
		\item if $\tau \in \RA(C)$, the twist $C^{\tau}$ over $\Fpt$ has $f_\tau = (x^2 + 1)^2$, so is the base change of a type $b = 0$ surface over $\Fp$,
		\item if $\rho \in \RA(C)$, the twist $C^{\rho}$ over $\Fpt$ has $f_\rho = (x^2 + x + 1)^2$, so it is the base change of a type $b = \pm 1$ surface over $\Fp$.
	\end{itemize}%
\end{proposition}%
\noindent From this we find that twisted superspecial curves and Jacobians are very sparse,
as in general Jacobians with non-trivial automorphisms beyond $[-1]$
are sparse in the 3-dimensional space of generic Jacobians:
\begin{itemize}
	\item $\tau \in \RA(C)$ implies $\RA(C) \cong \ZZ_2 \times \ZZ_2$, $\RA(C) \cong S_4$, or $\RA(C) \cong D_{2 \times 6}$
	which implies $\J$ is $(2,2)$-isogenous to a square $E^2$
	for some curve $E$~\cite{atlas}. This describes a 1-dimensional space of Jacobians.
	\item $\rho \in \RA(C)$ implies $\RA(C) \cong S_3$, $\RA(C) \cong S_4$, or $\RA(C) \cong D_{2 \times 6}$, which implies $\J$ is $(2,2)$-isogenous to a product $E \times E'$
	for some curves $E$ and $E'$ that are 3-isogenous~\cite{atlas}. This describes a 1-dimensional space of Jacobians.
\end{itemize}

Thus, we can very quickly identify the twisted superspecial Jacobians by their automorphism groups, which can be computed efficiently
directly from the curve equation or most sets of invariants, such as Igusa or Clebsch invariants~\cite{bolza,shaska}.
Summarizing, we get the following.

\begin{theorem}\label{thm:twisting superspecial}
	Let $\J$ be the Jacobian over $\Fp$ of a hyperelliptic curve $C$ of genus 2. Then $\J$ is superspecial if and only if
	\begin{itemize}
		\item $\J$ is of type $b = \pm 2$,
		\item $\J$ is of type $b = \pm 1$ and $\RA(C)$ is equal to $S_3$, $S_4$, or $D_{2 \times 6}$,
		\item $\J$ is of type $b = \phantom{-}0$ and $\RA(C)$ is equal to $\ZZ_2 \times \ZZ_2$, $S_4$, or $D_{2 \times 6}$.
	\end{itemize}
\end{theorem}
\noindent For completeness, we perform the same computation and characterization of twisted superspecial Jacobians over $\Fpt$
in \Cref{app:twisted superspecials}.

\section{Probabilistic Supersingularity Testing}
\label{sec:testing}
In this section, we derive three main results.
First, in \Cref{sec:minmax}, we derive a simple test for minimality or maximality of abelian varieties of any dimension $g$.
Second, we use this result to derive a simple algorithm for
superspeciality testing of Jacobians over $\Fp$ and $\Fpt$ in \Cref{sec:superspecial-testing}.
Third, in \Cref{sec:probabilistic-algorithm}, we specialize to Jacobians over $\Fp$ and
derive a Monte Carlo algorithm for supersingularity testing of Jacobians $\J$ over $\Fp$ (\Cref{alg:supersingularity_test}).
This last algorithm requires two specific results about Jacobians over $\Fp$:
\begin{enumerate}
    \item \Cref{thm:supersingular-orders}: A supersingular Jacobian over $\Fp$ has order $n = p^2 - bp + 1$
    for $b \in \{-2, -1, 0, 1, 2\}$. Any abelian surface over $\Fp$ such that
    $|A(\Fpt)| = n^2$ is supersingular.
    \item \Cref{thm:probabilistic-testing}:
    For an abelian surface over $\Fp$, we can verify $|A(\Fpt)| = n^2$
    by verifying $[n]P = \zero$ for random points $P$.
    This has only a negligible chance of false positives
    using a generalization of Sutherland's argument using
    the width of the Hasse--Weil interval.
\end{enumerate}
We proved the first result in \Cref{sec:unique-orders}, which allows us to derive the second result in \Cref{sec:killing-arguments}.
Together, these give \Cref{alg:supersingularity_test}.

\subsection{Minimality and Maximality Testing for Varieties}\label{sec:minmax}
In this section, we give an algorithm to verify minimality and maximality of abelian varieties of any dimension $g$
over $\Fp$ or $\Fpt$, by generalising the algorithm for our key example (\Cref{alg:key-example}) using \Cref{lem:width-bound-argument}.
Similar arguments are used in later sections when we derive supersingularity tests for Jacobians over $\Fp$,
and superspeciality tests over $\Fp$ and $\Fpt$.

\begin{lemalg}\label{alg:minmax-test}
    Let $A$ be an abelian variety of dimension $g$ over $\Fq$ with $q = p$ or $p^2$.
    Let $P \rgets A(\Fpt)$.
    \begin{itemize}
        \item If $[p]P = \pm P$, return \textbf{\upshape true}.
        \item Otherwise, return \textbf{\upshape false}
    \end{itemize}
\end{lemalg}
\Cref{alg:minmax-test} checks whether $A$ is minimal or maximal,
and thus whether it is superspecial by \Cref{lem:poonen en co}.
The failure probability is again negligible, as the following result shows.

\begin{theorem}\label{thm:minmax-test}
    \Cref{alg:minmax-test} runs in $O(\log p)$ time and returns \textbf{\upshape true}
    for any minimal or maximal abelian variety over $\Fq$. For any non-extremal variety of dimension $g$, \Cref{alg:minmax-test}
    returns \textbf{\upshape true} with probability at most $O(p^{-1})$.
\end{theorem}
\begin{proof}
    This follows directly from \Cref{lem:width-bound-argument}. We detail the maximal case; the minimal
    case follows similarly.
    For the subgroup $H = A[p+1](\Fpt)$, we get that $|H|$ divides $(p+1)^{2g}$,
    and thus if $|H|$ is larger than $4gp^{2g-1}$, then $|A(\Fpt)| = (p+1)^{2g}$ and $A$ is maximal,
    and therefore superspecial.
    So, for any non-maximal $A$, the subgroup $H$ is smaller than $O(4gp^{2g-1})$. The probability of sampling a point $P \in A(\Fpt)$ such that $[p+1]P = \zero$ is equal to the probability of sampling a point in $H$. In the non-maximal case we therefore have that the probability is
    \[
        \frac{|H|}{|A(\Fpt)|} \leq \frac{4gp^{2g - 1}}{(p - 1)^{2g}} = O(p^{-1}).
    \]
\end{proof}

\subsection{Superspeciality Testing for Jacobians}\label{sec:superspecial-testing}
By
combining \Cref{thm:minmax-test} with \Cref{thm:superspeciality if twist},
we get the following simple test for superspeciality
for Jacobians of any dimension $g$ over both $\Fp$ and $\Fpt$, as long
as $g$ is small compared to $p$.

\begin{lemalg}\label{alg:superspecial-test}
    Let $\J$ be a Jacobian of dimension $g$ of a curve $C$ over $\Fq$ with $q = p$ or $p^2$. Let $\Aut(C)$ be the geometric automorphism group of $C$ (i.e., automorphisms defined over $\Fpbar$).
    For each $\sigma \in \Aut(C)$:
    \begin{itemize}
        \item Let $\J^\sigma$ be the Jacobian of the twist $C^\sigma$ of $C$ by $\sigma$ over $\Fpt$.
        \item Sample $P \rgets \J^\sigma(\Fpt)$.
        \item If $[p]P = \pm P$, return \textbf{\upshape true}.
    \end{itemize}
    If no $\J^\sigma$ passes the $[p]P = \pm P$ test, return \textbf{\upshape false}.
\end{lemalg}
\Cref{alg:superspecial-test} checks
that $\J$ or any of its twists is maximal or minimal, which shows superspeciality by \cite[Thm. 5.3]{max_abvars}.
Using \Cref{thm:minmax-test}, we show that the failure probability is again negligible.

\begin{theorem}\label{thm:superspecial-test}
    Let $p > g + 1$. Then
    \Cref{alg:superspecial-test} runs in $O(g \log p)$ time and returns \textbf{\upshape true}
    for any superspecial Jacobian. For any non-superspecial Jacobian of dimension $g$, \Cref{alg:superspecial-test}
    returns \textbf{\upshape true} with probability at most $O((g-1) \cdot p^{-1})$.
\end{theorem}
\begin{proof}
    By \Cref{thm:minmax-test}, the probability of passing the $[p]P = -P$ test
    for any non-maximal Jacobian
    is bounded by $O(p^{-1})$.
    By Hurwitz's automorphisms theorem, we get $|\Aut(C)| \leq 84(g - 1)$, as long as
    $p > g + 1$~\cite{roquette}\footnote{
    \textcite{roquette} shows this result for all curves except one, but the exceptional curve is itself a superspecial curve.}
    and we can bound the number of twists we have to visit by the size of $\Aut(C)$
    by \Cref{lem:twists}.
    Each test costs $O(\log p)$ group operations, giving a total cost of $O(g \log p)$. Furthermore, if the $[p]P = -P$
    test succeeds on any twist $\J^\sigma$,
    then $\J$ is superspecial and a union bound over the twists gives the stated probability $O((g-1) \cdot p^{-1})$.
    The analogous argument holds replacing maximal with minimal and $[p]P = -P$ with $[p]P = P$.
\end{proof}

\subsection{Probabilistic Order Determination over $\Fp$}\label{sec:killing-arguments}
We now generalize Sutherland's probabilistic argument to cases beyond minimal or maximal abelian surfaces. If a uniformly random point $P \in A(\Fq)$ on an abelian variety $A$ of dimension $g$ over $\Fq$ satisfies
$[n]P = \zero_A$ for some $n \in \NN$, what can we say about the order of $A(\Fq)$?
We answer this question for abelian surfaces over $\Fp$ in  \Cref{thm:probabilistic-testing}:
for the supersingular orders $n = p^2 - bp + 1$,
if $[n]P = \zero_A$ for a random $P \in A(\Fp)$ and $[n]Q = \zero_A$
for a random $Q \in A^t(\Fp)$, then it is extremely likely that $A$ is supersingular (at least for large $p$).
A special case is $b = - 2$, which only needs a single test $[n]P = \zero_A$
for a random $P \in A(\Fp)$ to verify supersingularity.
We derive these results case by case, mainly using \Cref{thm:supersingular-orders}.
We rely heavily on \Cref{cor:lam-in-terms-phi}.\footnote{
    As in \Cref{sec:unique-orders}, the arguments here use only random sampling and scalar
multiplication, so no polarization is assumed in this subsection.
}

\paragraph{The extremal cases $\mathbf{b = \pm 2}$.}
Interestingly, these differ for $b = -2$ and $b = 2$.
The former needs only 1 point, whereas the latter needs 2.
In both cases, we multiply by $p\pm 1$ instead of the supersingular order $(p\pm1)^2$.
\begin{lemma}\label{lem:prob-p+1}
    Let $A$ be an abelian surface over $\Fp$. Let $P$ be a random point on $A(\Fp)$.
    The probability that $[p+1]P = \zero$ is
    \begin{itemize}
        \item 1, if $A$ is supersingular of type $b = -2$,
        \item $O(p^{-1/2})$, otherwise.
    \end{itemize}
\end{lemma}
\begin{proof}
    Let $m = p+1$.
    The probability to sample a point $P$ with $[m]P = \zero$ equals
    \[
        \mathcal{P} := \frac{|A(\Fp)[m]|}{|A(\Fp)|}.
    \]
    Let $h_0, \ldots, h_3 \in \NN$ be such that $A(\Fp)[m] \cong \Z_{h_0} \times \Z_{h_1} \times \Z_{h_2} \times \Z_{h_3}$ and $h_{i+1} \mid h_i$.
    From \Cref{lem:weil-pairing},
    $h_3 \mid h_2 \mid 2$,  and we have $h_1 \mid h_0 \mid p+1$ and $h_1 \mid \lambda_0$, $h_1 \mid \lambda_1$.
    Since  $\lambda_0 = (p+1)(p-1+a_p) + 2-b_p$ and $ \lambda_1 = -a_p(p+1) -2a_p -2(2-b_p)$, this implies $h_1 \mid 2a_p$.

    Let $W = 8\sqrt{p}(p+1)$ denote the width of the
    Hasse--Weil interval over $\mathbb{F}_p$.
    If $h_0 h_1 \leq W$, then $|A(\mathbb{F}_p)[m]| = h_0h_1h_2h_3 \leq 4W$, and therefore
    \[
    \mathcal{P} \leq \frac{4W}{(\sqrt{p}-1)^4} = \frac{32\sqrt{p}(p+1)}{(\sqrt{p}-1)^4} = O\big(p^{-1/2}\big).
    \]
    Suppose instead that $h_0 h_1 > W$.
    If $a_p \neq 0$, then $h_1 \leq 2|a_p| \leq 8\sqrt{p}$ by~\Cref{eq:bound-a}, so that
    $h_0 h_1 \leq 8\sqrt{p}(p+1) = W$, a contradiction. Hence, $a_p = 0$.
    Now $A(\mathbb{F}_p)$ contains a subgroup isomorphic to $\mathbb{Z}_{h_0} \times \mathbb{Z}_{h_1}$,
    whose order exceeds $W$ and divides $(p+1)^2$. By \Cref{lem:width-bound-argument}, $|A(\mathbb{F}_p)| = (p+1)^2$ and so $b_p = -2p$. Therefore, $A$ is supersingular of type $b = -2$ and $\mathcal{P} = 1$.
\end{proof}

\begin{lemma}\label{lem:prob-p-1}
    Let $A$ be an abelian surface over $\Fp$.
    The probability that
    $[p-1]P = \zero$ for a random $P \in A(\Fp)$
    and $[p-1]Q = \zero$ for a random $Q \in A^t(\Fp)$,
    is
    \begin{itemize}
        \item 1, if $A$ is supersingular of type $b = 2$,
        \item $O(p^{-1/2})$, otherwise.
    \end{itemize}
\end{lemma}
\begin{proof}
    Let $G = A[p-1](\Fp)$ and $H = A^t[p-1](\Fp)$ with twisting isomorphism $\rho: A^t \to A$.
    Then $S = \langle G, H^\rho \rangle$
    is a subgroup of $A(\Fpt)$ whose order divides $\#A[p-1] = (p-1)^4$.
    By the width-bound argument, if $|S| \geq 8 (p^3 + p)$, then $|A(\Fpt)| = (p-1)^4$,
    which implies $A$ is minimal, and hence superspecial by \Cref{lem:poonen en co}.
    In particular, it is of type $b = 2$
    as this is the only $\Fp$-type to lift to a minimal abelian surface over $\Fpt$.
    Thus, in the non-supersingular case we have $|S| \leq 8 (p^3 + p)$.
    Since $G$ and $H$ share only the 2-torsion, of size at most 16, and $|S| = |G| \cdot |H|/|G \cap H|$, then
    $|G| \cdot |H| \leq 16 \cdot |S| \leq 2^7 \cdot p^3 + 2^7 \cdot p$, so at
    least one of $|G|$ or $|H|$ is smaller than $2^{7/2} \cdot p^{3/2} + \epsilon$.
\end{proof}

\begin{remark}
    We stress that the situation for $n = p-1$ is significantly
    different from the situation for $n = p+1$. Indeed, there exist ordinary
    abelian surfaces $A/\Fp$ such that $[p-1]A(\Fp) = \zero$.
    An example is given by the Jacobian $\J_C$
    over $\Fp$, where $p = 73$, for the hyperelliptic curve
    \[
        C : y^2 = x^6 + 68 x^5 + 41 x^4 + 31 x^3 + 17 x^2 + 11 x + 51,
    \]
    whose rational group $\J_C(\Fp)$ is isomorphic to
        \[
        \J_C(\Fp) \cong \ZZ/12 \times \ZZ/12 \times \ZZ/6 \times \ZZ/6.
    \]
        Crucially, the proof of \Cref{lem:prob-p+1} uses the shape of the Weil bounds
        (\Cref{eq:bound-a})
        to make the probability for non-zero $a$ values negligible; similar arguments cannot be repeated for $p - 1$.
\end{remark}

\paragraph{The cyclic cases $\mathbf{b = 0, \pm 1}$.}
All three cases are very similar. We show this for $n = p^2 + p + 1$, the other cases follow mutatis mutandis.

\begin{lemma}\label{lem:prob-p2-p-1}
    Let $A$ be an abelian surface over $\Fp$ and
    let $n = p^2 - bp + 1$ with $b \in \{-1, 0, 1\}$.
    The probability that
    $[n]P = \zero$ for a random $P \in A(\Fp)$
    and $[n]Q = \zero$ for a random $Q \in A^t(\Fp)$,
    is
    \begin{itemize}
        \item 1, if $A$ is supersingular of type $b =0, \pm 1$,
        \item $O(p^{-1/2})$, otherwise.
    \end{itemize}
\end{lemma}
\begin{proof}
    Let $H = A[n](\Fp)$ and $G = A^t[n](\Fp)$. These groups do not share points
    beyond the 2-torsion. We again look at $S = \langle H, G^\rho \rangle$,
    where $\rho$ is the twisting isomorphism.
    Let $\gcd(n, p^2 -1) = d \in \{1, 2, 3\}$. By \Cref{lem:weil-pairing},
    beyond the 2-torsion or 3-torsion, the rank of $A[n](\Fpt)$ is at most 2, so $|S|$ divides $d^2 n^2$. Suppose first that $d = 1$ so that $|S|$ divides $n^2$.
    By the width-bound argument, $|S| \geq 8(p^3 + p)$ implies $|A(\Fpt)| = n^2$, hence $A$ is
    supersingular by \Cref{thm:supersingular-orders}. If $d = 2$ or $3$, then we replace $S$ by its prime-to-$d$ part, which has order dividing $n^2$, and the same argument applies (with the probability bound only losing a constant factor of at most $81$).
    We conclude that either $|S| \geq 8(p^3 + p)$ hence $A$ is supersingular,
    or $|S| \leq 8(p^3 + p)$ and so the $[n]P = \zero$ test succeeds on both $A$
    and $A^t$ with probability $O(p^{-1/2})$, similarly as in the proof of~\Cref{lem:prob-p-1}.
\end{proof}

\paragraph{General Theorem.}
We summarize the results in this section in the following theorem,
which follows directly from \Cref{lem:prob-p+1,lem:prob-p-1,lem:prob-p2-p-1}.

\begin{theorem}\label{thm:probabilistic-testing}
    Let $n = p^2 - bp + 1$ for $b \in \{ \pm 1, 0 \}$, and $n = p \pm 1$ for $b = \mp 2$. Let $A$ be an abelian surface over $\Fp$, $P$ be a random point of $A(\Fp)$ and $Q$ be a random point of $A^t(\Fp)$.
    Then,
    \begin{itemize}
        \item if $b = -2$, the probability that $[n]P = \zero$ is 1 if $A$ is supersingular, and $O(p^{-{1/2}})$ otherwise.
        \item if $b \in \{-1, 0, 1, 2\}$,
        the probability that $[n]P = \zero$ \emph{and} $[n]Q = \zero$
        is 1 if $A$ is supersingular, and $O(p^{-{1/2}})$ otherwise.
    \end{itemize}
\end{theorem}

\subsection{Supersingularity Testing for Abelian Surfaces over $\Fp$}\label{sec:probabilistic-algorithm}
Based on \Cref{thm:probabilistic-testing},
we derive a simple algorithm to verify supersingularity for abelian surfaces over $\Fp$
in \IsSupersingular.

\begin{algorithm}[ht!]
	\caption{\DEFIsSupersingular}
    \label{alg:supersingularity_test}
    \begin{algorithmic}[1]
        \REQUIRE{An abelian surface $A$ defined over $\Fp$.}
        \ENSURE{A boolean indicating whether $A$ is supersingular. If true, we also return $|A(\Fp)|$.}
    \end{algorithmic}
    \begin{minipage}[t]{0.48\linewidth}
	\begin{algorithmic}[1]
        \STATE $P \picks A(\Fp) \setminus (A[2] \cup A[3])$
        \STATE $Q \picks A^t(\Fp) \setminus (A^t[2] \cup A^t[3])$
        \STATE $R \gets [p]P$
        \IF{$R = -P$}
            \RETURN {\tt true}, $(p + 1)^2$
        \ENDIF
        \IF{$R = P$}
            \STATE $S \gets [p]Q$
            \IF{$S = Q$}
            \RETURN {\tt true}, $(p - 1)^2$
            \ENDIF
        \ENDIF
    \end{algorithmic}
\end{minipage}%
\hfill%
\begin{minipage}[t]{0.48\linewidth}
    \begin{algorithmic}[1]
        \setcounter{ALC@line}{9}
        \STATE $T = [p]R$
        \IF{$T = -P$}
            \STATE $n \gets p^2 + 1$
        \ELSIF{$T + P = \pm R$}
            \STATE $n \gets p^2 \mp p + 1$
        \ELSE
            \RETURN {\tt false}
        \ENDIF
        \IF{$[n]Q = \zero$}
            \RETURN {\tt true}, $n$
        \ENDIF
        \RETURN {\tt false}
    \end{algorithmic}
    \end{minipage}
\end{algorithm}

\begin{theorem}\label{thm:main-theorem-probabilistic}
    \Cref{alg:supersingularity_test} runs in $O(\log p)$ time and returns \textbf{\upshape true} for any supersingular abelian surface.
    For any non-supersingular abelian surface, \Cref{alg:supersingularity_test} returns \textbf{\upshape true} with probability $O(p^{-1/2})$.
   \end{theorem}
\begin{proof}
    As multiplication-by-$p$ and random sampling can be performed in $O(\log p)$ time, \Cref{alg:supersingularity_test} terminates in $O(\log p)$.
    Correctness and probability bounds follow directly from \Cref{thm:probabilistic-testing}, as long as
    we exclude 2- or 3-torsion.
    Although cryptographically such torsion occurs with negligible probability,
    their exclusion ensures that \Cref{alg:supersingularity_test} never misclassifies a supersingular surface.\footnote{
    For example, as $\gcd(p^2 + p + 1, p - 1) = 3$ when $p \equiv 1 \bmod 3$, such points could
    misclassify a supersingular Jacobian of type $b = -1$ as one of type $b = 2$.}
\end{proof}
The probability that \IsSupersingular returns {\tt true} when $A$ is ordinary is negligible when $p$ is large.
For example, primes $p$ in isogeny-based cryptography are often of size $2 \lambda$ bits,
making the failure probability of \IsSupersingular roughly $2^{-\lambda}$ where $\lambda$ is the security parameter,
typically $\lambda \in \{128, 192, 256\}$.

\begin{remark}
    It is crucial that the points sampled in \Cref{alg:supersingularity_test} are uniformly random,
    otherwise an adversary may always find specific points on specific abelian surfaces that allow an ordinary
    abelian surface to pass this test. For validation purposes in cryptographic applications, we must therefore
    assume that an adversary may choose the abelian surface, but not the randomness.
\end{remark}

\begin{remark}
    We stress that the use of the twist allows the full algorithm to work over $\Fp$ instead of $\Fpt$.
    Although this doesn't affect asymptotic cost estimates, this significantly affects real-world performance.
\end{remark}

\section{Conclusive Supersingularity Testing over $\Fp$}
\label{sec:conclusive}
In this section, we discuss conclusive variants of supersingularity testing
for abelian surfaces over $\Fp$, in contrast to the probabilistic methods from \Cref{sec:testing},
as some applications may require \emph{proof} that the Jacobian is supersingular.
Similarly, for small characteristics $p$, \Cref{alg:supersingularity_test} may have
a non-negligible failure rate, and conclusive testing ensures certainty.

A na\"ive approach simply verifies the explicit order of points $P \in \J(\Fp)$ and $Q \in \J^t(\Fp)$,
as a generalization of \cite[Alg. 1]{CSIDH}, using a product tree for additional efficiency.
As shown in \cite{pairings-CSIDH} for elliptic curves, the use of pairings significantly improves this method.
We generalize \cite[Alg. 5]{pairings-CSIDH} to Jacobians of genus-2 curves.
Due to the existence of isotropic subgroups for $g > 1$, this generalization is non-trivial
and requires a precise understanding of isotropy conditions on $\J(\Fp)$,
explored in \Cref{sec:isotropy-conditions-min-maxc,sec:isotropy-conditions-other} using results from
\Cref{sec:pairings-jacobians}. We furthermore analyze the requirements of discrete logarithms
in $\mu_{p+1}$ in verifying supersingularity.
Lastly,
in \Cref{sec:cryptographic-primes}, we
discuss a specialized variant for Jacobians where $2^f$ divides the exponent for a large power of $f$,
common in cryptographic applications, which allows a significant speed-up~\cite[Alg. 2]{dCSIDH}.

\subsection{Isotropy on Supersingular Abelian Surfaces of type $b = \pm 2$}\label{sec:isotropy-conditions-min-maxc}

For the conclusive supersingularity verification, we will perform specific pairings in $A[m]$ for specific values of $m \in \ZZ_{>0}$.
Depending on the isotropy conditions of these subgroups, the resulting pairing can be computed
over $\Fpk$ or $\Fpkt$ where $k$ is the embedding degree of $m$.
We prove this case by case per $b$-type, starting with $b = \pm 2$, where we have a large rank-2 subgroup of $A[m]$,
before we discuss the cyclic cases $b \in \{-1, 0, 1\}$.

\begin{lemma}\label{lem: perfect pairings on type b = 2}
	Let $A$ be a supersingular abelian surface over $\Fp$ of type $b = 2$
  and let $m = p-1$ or $m = \frac{p-1}{2}$ be the exponent of $A(\Fp)$ and let $n = p-1$.
  \begin{enumerate}[label=(\roman*)]
    \item The subgroup generated by the eigenspaces $A[\pi - 1] \subseteq A(\Fp)$
    and $A[\pi + 1] \subseteq A^t(\Fp)^\rho$ is of index 4 in $A[m] \subseteq A(\Fpt)$ if
    $A(\Fp) \cong (\ZZ/m\ZZ)^2$, of index 8 if
    $A(\Fp) \cong \ZZ/m\ZZ \times \ZZ/\tfrac{m}{2}\ZZ \times \ZZ/2\ZZ$, and of index 1 if
    $A(\Fp) \cong (\ZZ/m\ZZ \times \ZZ/2\ZZ)^2$.
    \item  For every
        $Q \in A(\Fp)$ of even order $k$, there exists $P \in A(\Fp)$ such that
        $e_n(P, Q)$ has order $k/2$.
  \end{enumerate}
\end{lemma}
\begin{remark}
  Note that, according to \Cref{thm: possible group structures of rational jacobians}, $p \equiv 3 \bmod 4$
  implies that the exponent is always $m = p-1$,
  whereas when $p \equiv 1 \bmod 4$ both $m = p - 1$ and $m = \tfrac{p-1}{2}$ can happen.
\end{remark}

\begin{proof}
  \textbf{(i)} This follows directly from \Cref{thm: possible group structures of rational jacobians}
  and \Cref{lem:decomposition of Jn}.

  \textbf{(ii)} As $p + 1 \equiv 2 \pmod{n}$,~\Cref{lem:trace-pairing} gives
$e_n((\pi + 1)R, Q) = e_n(R, Q)^2$ for all $R \in A[n]$ and $Q \in A(\Fp)$.
Let $Q \in A(\Fp)$ have even order $k$. By non-degeneracy of $e_n$ on $A[n]$, the
homomorphism $e_n(\cdot, Q) \colon A[n] \to \mu_n$ has image $\mu_k$ and so there is $R \in A[n]$ with $e_n(R, Q)$ of order $k$. Then $P = (\pi + 1)R \in A(\Fp)$ satisfies that
$e_n(P, Q) = e_n(R, Q)^2$ has order $k/2$.
\end{proof}
\Cref{lem: perfect pairings on type b = 2} shows that abelian surfaces of type $b = 2$
are as non-isotropic as possible, except for the 2-torsion.
The case $b = -2$ is the exact opposite.

\begin{lemma}\label{lem: trivial pairings on type b = -2}
	Let $A$ be a supersingular abelian surface over $\Fp$ of type $b = -2$
  and let $m = p + 1$ or $m = \tfrac{p + 1}{2}$ be the exponent of $A(\Fp)$.
  \begin{enumerate}[label=(\roman*)]
    \item The subgroup generated by eigenspaces $A[\pi - 1] \subseteq A(\Fp)$
    and $A[\pi + 1] \subseteq A^t(\Fp)^\rho$ is of index 4 in $A[m] \subseteq A(\Fpt)$ if
    $A(\Fp) \cong (\ZZ/m\ZZ)^2$, of index 8 if
    $A(\Fp) \cong \ZZ/m\ZZ \times \ZZ/\tfrac{m}{2}\ZZ \times \ZZ/2\ZZ$, and of index 1 if
    $A(\Fp) \cong (\ZZ/m\ZZ \times \ZZ/2\ZZ)^2$.
    \item The pairing $e_n$ is trivial on $A(\Fp)$ and on $A^t(\Fp)^\rho$.
  \item The pairing
        \[
          A(\Fp) \times A^t(\Fp) \to \mu_n, \qquad
          (P, Q) \mapsto e_n(P, Q^\rho),
        \]
        has left kernel $A[2](\Fp)$ and right kernel $A^t[2](\Fp)$.
    \item Let $P_1, P_2 \in A(\Fp)$ and $Q_1, Q_2 \in A^t(\Fp)$, let $\zeta$ be a
        primitive $n$-th root of unity, and let
        $M = \big(\log_\zeta e_n(P_i, Q_j^\rho)\big)_{i,j}$ over $\ZZ/n\ZZ$. Then
        \[
          |\mathrm{im}\, M| \;\geq\;
          \frac{|\langle P_1, P_2 \rangle| \cdot |\langle Q_1, Q_2 \rangle|}{4n^2}.
        \]
  \end{enumerate}
\end{lemma}
\begin{proof}
  \textbf{(i)} This follows directly from \Cref{thm: possible group structures of rational jacobians}
  and \Cref{lem:decomposition of Jn}.

  \textbf{(ii)} Write $G = A(\Fp) = A[\pi - 1]$ and
  $G^t = A^t(\Fp)^\rho = A[\pi + 1]$.
  As $p + 1 \equiv 0 \pmod{n}$,~\Cref{lem:trace-pairing} gives
  $e_n((\pi + 1)R, Q) = 1$ for all $R \in A[n]$ and $Q \in G$, and
  $e_n((\pi - 1)R, S) = 1$ for all $R \in A[n]$ and $S \in G^t$. Since $(\pi + 1)A[n] = G$ and
  $(\pi - 1)A[n] = G^t$, both $G$ and $G^t$ are isotropic.

    \textbf{(iii)} Since $G$ is isotropic of order $n^2$ in $A[n]$, we have that $G = G^{\perp}$, and similarly for $G^t$. Thus, a point $P \in G$ lies in the left kernel if and only if $P \in G^t$, i.e., $P \in G \cap G^t = A[2](\Fp)$. By symmetry, the right kernel is $A^t[2](\Fp)$.

    \textbf{(iv)} Write $H = \langle P_1, P_2 \rangle \subseteq G$ and
    $H^t = \langle Q_1^\rho, Q_2^\rho \rangle \subseteq G^t$.
    The image of $M$ is the image of the homomorphism
    $H \to \mathrm{Hom}(H^t, \mu_n)$, $P \mapsto e_n(P, \cdot)$.
    This is a composition of maps $H \to \mathrm{Hom}(G^t, \mu_n)$ with the restriction
    $\mathrm{Hom}(G^t, \mu_n) \to \mathrm{Hom}(H^t, \mu_n)$. The first map has kernel $H \cap A[2](\Fp) = H[2]$ by \textbf{(iii)}. The order of the kernel is therefore at most $4$ (since $H$ is generated by two elements), and the image has order at least $|H|/4$.
    The restriction map has kernel of order $[G^t : H^t] = n^2/|H^t|$. Hence,
    $|\mathrm{im}\, M| \geq |H| \cdot |H^t| / (4n^2)$.
\end{proof}

\begin{remark}\label{rem:2-rank from f}
  Which of the three cases in \textbf{(i)} of \Cref{lem: perfect pairings on type b = 2} and
  \Cref{lem: trivial pairings on type b = -2} applies is governed by the $2$-rank of $A(\Fp)$ alone:
  by the index formula of \Cref{lem:decomposition of Jn}, the three group structures listed there are
  precisely those of $2$-rank $2$, $3$ and $4$, giving index $4$, $8$ and $1$ respectively.
  When $A$ is the Jacobian of a genus-$2$ curve $C : y^2 = f(x)$, this rank is read off directly
  from the factorization type of $f$ by \Cref{lem: factoring fx gives 2 rank}, so the case distinction
  costs a single polynomial factorization rather than a group-structure computation.
  Moreover, \Cref{thm: possible group structures of rational jacobians} forces the $2$-rank to be at
  least $2$ when $b = \pm 2$, so a factorization type of $f$ of $2$-rank $0$ or $1$ already certifies
  that $\J$ is \emph{not} supersingular of type $b = \pm 2$.
\end{remark}
  The proof of supersingularity for the case $b = 2$ is therefore very efficient: we only need to verify the order of pairing values $\zeta \in \Fp$,
  and both the pairing and the verification can be computed over $\Fp$. This is because in this case, the eigenspaces $1$ and $p$,
  resp. $-1$ and $-p$ collapse, so that the pairing on $A(\Fp)$ is as
  non-degenerate as possible.

For abelian surfaces of type $b = -2$ the situation is reversed: for $n = p + 1$, the pairing
  $e_n$ is trivial on $A(\Fp)$. Pairings between
  points of $A(\Fp)$ therefore carry no information, and we instead pair points of
  $A(\Fp)$ with points of $A^t(\Fp)^\rho$. A single pairing $e_n(P, Q^\rho)$ of
  order $k$ only shows that $A(\Fp)$ and $A^t(\Fp)$ contain cyclic subgroups of
  order $k$, which is too small for the width-bound argument. We therefore take
  $P_1, P_2 \in A[n](\Fp)$ and $Q_1, Q_2 \in A^t[n](\Fp)$ and compute the four
  pairings $\zeta_{i,j} = e_n(P_i, Q_j^\rho)$. We can compute a matrix $M = (\log_\zeta \zeta_{i,j})$ with discrete logarithms in $\mu_n$\footnote{For most
  applications in isogeny-based cryptography, smoothness of $p+1$ is a common assumption, thus
  discrete logarithms are efficient.}, and then determine the size of its image. By~\Cref{lem: trivial pairings on type b = -2}\textbf{(iv)}, points generating large subgroups of $A(\Fp)$ and $A^t(\Fp)$ give a large image for each of the three possible group structures.

\begin{corollary}\label{cor:conclusive-testing-non-cyclic}

  Let $A$ be an abelian surface over $\Fp$. Write
  $W = 8\sqrt{p}(p+1)$, the width of the Hasse--Weil interval over $\Fp$.
  \begin{itemize}
    \item Let $n = p-1$, and let $P_1, P_2 \in A[n](\Fp)$ and $Q_1, Q_2 \in A^t[n](\Fp)$. If $e_n(P_1, P_2)$ and $e_n(Q_1, Q_2)$ have orders $k, k'$ with $k^2, k'^2 > W$, then $A$ is supersingular of type $b = 2$.
    \item  Let $n = p + 1$, let $\zeta$ be a primitive $n$-th root of unity, and let $P_1, P_2 \in A[n](\Fp)$ and $Q_1, Q_2 \in A^t[n](\Fp)$. Let $\zeta_{i,j} = e_n(P_i, Q_j^\rho)$, where $\rho$ is the twisting isomorphism. If the matrix $M = (\log_\zeta \zeta_{i,j})_{i,j}$ over $\mathbb{Z}/n\mathbb{Z}$ satisfies $|\mathrm{im}\, M| > W$, then $A$ is supersingular of type $b = -2$.
  \end{itemize}
  Conversely, let $A$ be supersingular of type $b = \pm 2$ and let $n = p \mp 1$.
  For $p$ large enough, there exist $P_1, P_2 \in A(\Fp)$ and $Q_1, Q_2 \in A^t(\Fp)$
  satisfying the corresponding condition above.
\end{corollary}

\begin{proof}
First let $n = p - 1$. If $e_n(P_1, P_2)$ has order $k$ then
$|\langle P_1, P_2 \rangle| \geq k^2 > W$. Moreover, $|\langle P_1, P_2 \rangle|$ divides $n^2$, so $|A(\Fp)| = (p-1)^2$ by~\Cref{lem:width-bound-argument}. The same argument applied to $Q_1, Q_2$ gives $|A^t(\Fp)| = (p-1)^2$. It follows that $A$ is supersingular of type $b = 2$.

Now let $n = p + 1$. We have that $|\mathrm{im}\, M| \leq |\langle P_1, P_2 \rangle|$ and, as the images of $M$ and $M^T$ have the same size, we have $|\mathrm{im}\, M| \leq |\langle Q_1^\rho, Q_2^\rho \rangle| = |\langle Q_1, Q_2 \rangle|$.
Both orders divide $n^2 = (p+1)^2$ and (by hypothesis) $|\mathrm{im}\, M| > W$, so~\Cref{lem:width-bound-argument} gives $|A(\Fp)| = |A^t(\Fp)| = (p+1)^2$.

For the converse, note that by \Cref{thm: possible group structures of rational jacobians}, each of the three possible group structures
of $A(\Fp)$ contains a subgroup generated by two elements of order at least $n^2/4$,
and the same holds for $A^t(\Fp) \cong A(\Fp)$. Moreover,
$A(\Fp)$ contains a point of order at least $n/2$.
Let $A$ be of type $b = 2$, so $n = p - 1$, and let $Q \in A(\Fp)$ be of maximal order
$k_0 \in \{n, n/2\}$. By \Cref{lem: perfect pairings on type b = 2}\textbf{(ii)},
there is $P \in A(\Fp)$ with
$e_n(P, Q)$ of order $k = k_0/2 \geq n/4$, and similarly on $A^t$ of order $k' \geq n/4$.
Now $k^2 \geq \left(\tfrac{p-1}{4}\right)^2 > W$ for $p > 2^{18}$.
A similar argument holds for type $b = -2$.
If moreover
$A(\mathbb{F}_p) \cong (\mathbb{Z}/n\mathbb{Z})^2$,
which by \Cref{rem:2-rank from f} can be read off from the 2-rank of $A(\mathbb{F}_p)$,
then any bases give $k = k' = n/2$ for $b = 2$, resp. $|\mathrm{im}\,M| \ge n^2/4$ for $b = -2$
by \Cref{lem: trivial pairings on type b = -2}\textbf{(iv)},
and the conditions hold for $p > 2^{10}$.
\end{proof}

\begin{remark}
  The converse statement ensures that such a test passes for supersingular surfaces.
  For cryptographic applications, primes always satisfy these bounds as they are in the range
  of hundreds of bits.
  In both cases ($b = \pm 2$), product trees and Lucas sequences can be used to speed up the verification~\cite{pairings-CSIDH}, and many operations can remain over $\Fp$ as both points lie on the same $\Fp$-rational Kummer surface.
\end{remark}

\subsection{Isotropy on Cyclic Supersingular \texorpdfstring{Abelian Surfaces}{Abelian Surfaces}}\label{sec:isotropy-conditions-other}
Due to their cyclic structure, similar pairing-based techniques for abelian surfaces of type $b \in \{-1, 0, 1\}$
would immediately verify the full group structure of $A(\Fp)$ and $A^t(\Fp)$, i.e.,
if $P \in A(\Fp)$ and $Q \in A^t(\Fp)$ have a pairing of order $n$, this verifies both
$A(\Fp)$ and $A^t(\Fp)$ have a subgroup of order $n$.
Unfortunately, such pairings will always be trivial,
as in general the $n$-th roots of unity only live in $\Fpthree$, $\Fpf$ or $\Fps$.
To perform conclusive supersingularity verification, we must therefore extend the field we work in.

\begin{lemma}[The case $b = 0$]\label{lem:b0}
  Let $A$ be an abelian surface over $\Fp$ of type $b = 0$
  and let $m = p^2 + 1$ be the exponent of $A(\Fp)$.

  \begin{enumerate}[label=(\roman*)]
    \item The subgroup generated by the four eigenspaces
    $A[\pi - 1] \subseteq A(\Fp)$,
    $A[\pi + 1] \subseteq A^t(\Fp)^\rho$,
    $A[\pi - p],
    A[\pi + p] \subseteq A^t(\Fpt)^{\rho'}$
    generate a subgroup of index 8 of the
    full $m$-torsion $A[m] \subseteq A(\Fpf)$.
    \item Let $P \in A[\pi - 1]$ and $Q \in A[\pi + 1]$, then $e_m(P, Q)$ is trivial.
    \item For any generators $P, Q, R, S$ of the four eigenspaces (see above), the pairings $e_m(P,R)$ and $e_m(Q,S)$ have order exactly $m/2$.
  \end{enumerate}
\end{lemma}

\begin{proof}
  \textbf{(i)} follows from~\Cref{thm: possible group structures of rational jacobians} and~\Cref{lem:decomposition of Jn}.
  Write $m=2m'$ with $m'$ odd (possible as $p^2\equiv1\bmod 8$). As in the proof of~\Cref{lem:decomposition of Jn}, $A(\Fp)$ is cyclic and every eigenspace meets $A[2]$ in
  $A[2](\Fp)=\langle T\rangle\cong\Z/2\Z$.
  \textbf{(ii)} By~\Cref{lem:almost symplectic basis}, $e_m(P,Q)\in\mu_2$, and its value is
  $e_2([m']P,[m']Q)=e_2(T,T)=1$.
  \textbf{(iii)} The same computation gives $e_2([m']P,[m']R)=1$, so $e_m(P,R)$ has order
  dividing $m'$. For the converse, note that $[A[m]:H]=8$ is a power of $2$, so
  $H\supseteq A[m']$ and the four eigenspaces restrict to a basis of
  $A[m']\cong(\Z/m'\Z)^4$. By~\Cref{lem:almost symplectic basis} the pairings
  $e_m(P,Q),e_m(P,S),e_m(Q,R),e_m(R,S)$ all have order dividing $2$, hence are trivial
  on $m'$-torsion. The Gram matrix of $e_{m'}$ on this basis therefore has only the
  entries $e_{m'}(P,R)$ and $e_{m'}(Q,S)$ off the diagonal, and non-degeneracy of
  $e_{m'}$ forces both to be primitive. Thus, $e_m(P,R)$ has order exactly $m'=m/2$, and likewise for $e_m(Q,S)$.
\end{proof}
The cases $b = \pm 1$ have similar issues: here $m = p^2 \mp p + 1$ is divisible by 3 exactly when
$p \equiv \mp 1 \bmod 3$, and in that case the eigenvalues $1, -1, p, -p$ are no longer distinct modulo 3,
resulting in pairings of order $m / 3$. However, this will still be enough
to give a proof of supersingularity.

\begin{lemma}[The case $b = 1$]\label{lem:b1}
    Let $A$ be an abelian surface over $\Fp$ of type $b = 1$ and let $m = p^2 - p + 1$
    be the exponent of $A(\Fp)$.

    \begin{enumerate}[label=(\roman*)]
      \item
      The subgroup generated by the four eigenspaces
      $A[\pi - 1] \subseteq A(\Fp)$,
      $A[\pi + 1] \subseteq A^t(\Fp)^\rho$,
      $A[\pi - p] \subseteq A^t(\Fpthree)^{\rho}$, and
      $A[\pi + p] \subseteq A(\Fpthree)$
      generate a subgroup of index 9 if $p \equiv 2 \bmod 3$, and of index 1 otherwise,
      of the full $m$-torsion $A[m] \cong (\ZZ/m\ZZ)^4 \subseteq A(\Fps)$.
      \item  Let $P \in A[\pi - 1]$ and $Q \in A[\pi + 1]$. Then $e_m(P, Q)$ is trivial or has order 3.
      \item For any generators $P, Q, R, S$ of the four eigenspaces above, the pairings $e_m(P, R)$ and $e_m(Q, S)$ have order $m/3$ or $m$.
    \end{enumerate}
\end{lemma}

\begin{proof}
  \textbf{(i)} This follows directly from \Cref{thm: possible group structures of rational jacobians}
  and \Cref{lem:decomposition of Jn}. Note $3 \mid m$ if and only if $p \equiv 2 \bmod 3$ and $m$ is odd. \textbf{(ii).} Since $\gcd(m, 2) = 1$ and $\gcd(m, p+1) \leq 3$
  by \Cref{lem:almost symplectic basis}, we get $e_m(P, Q) \in \mu_3$. \textbf{(iii).} Write $m = 3^em'$ with $3 \nmid m'$. Arguing as in~\Cref{lem:b0}(iii), the index in \textbf{(i)} is a power of $3$, so the four eigenspaces restrict to a basis of $A[m']$. By~\Cref{lem:almost symplectic basis}, the pairings $e_m(P,Q)$, $e_m(P,S)$, $e_m(Q,R)$, $e_m(R,S)$ all have order dividing $3$, hence are trivial on $m'$-torsion. The non-degeneracy of $e_{m'}$ forces $e_{m'}(P,R)$ to be primitive. The $3$-part is not determined: modulo $3$ the eigenvalues $\pm 1$, $\pm p$ collapse in pairs, and $e_3$ need not vanish between the two resulting order $3$ eigengroups.\footnote{
  Experiments suggest the $3$-part of $e_m(P,R)$ and
  $e_m(P,Q)$ is trivial, i.e., that $e_m(P,R)$ has order exactly
  $m/3$;~\Cref{cor:conclusive-testing-cyclic} does not depend on this.} If $3 \nmid m$, then $e = 0$ and the order is $m$. A similar argument holds for $e_m(Q,S)$.
\end{proof}

\begin{lemma}[The case $b = -1$]\label{lem:bm1}
  Let $A$ be an abelian surface over $\Fp$ of type $b = -1$ and let $m = p^2 + p + 1$
  be the exponent of $A(\Fp)$.

  \begin{enumerate}[label=(\roman*)]
    \item
    The subgroup generated by the four eigenspaces
    $A[\pi - 1] \subseteq A(\Fp)$,
    $A[\pi + 1] \subseteq A^t(\Fp)^\rho$,
    $A[\pi - p] \subseteq A(\Fpthree)$, and
    $A[\pi + p] \subseteq A^t(\Fpthree)^{\rho'}$
    generate a subgroup of index 9 if $p \equiv 1 \bmod 3$, and of index 1 otherwise,
    of the full $m$-torsion $A[m] \cong (\ZZ/m\ZZ)^4 \subseteq A(\Fps)$.
    \item  Let $P \in A[\pi - 1]$ and $Q \in A[\pi + 1]$. Then $e_m(P, Q)$ is trivial.
    \item For any generators $P, Q, R, S$ of the four eigenspaces above, the pairings $e_m(P, R)$ and $e_m(Q, S)$ have order $m/3$ or $m$.
  \end{enumerate}
\end{lemma}

\begin{proof}
  As for~\Cref{lem:b1}, where $\gcd(m, p+1) = 1$ gives triviality in \textbf{(ii)}.
\end{proof}

\begin{corollary}\label{cor:conclusive-testing-cyclic}
  Let $A$ be an abelian surface over $\Fp$, let $b \in \{-1, 0, 1\}$, let
  $n = p^2 - bp + 1$. Let
  $P \in A[n](\Fp)$, $Q \in A^t[n](\Fp)$ and
  $R, S \in A[n]$,\footnote{In practice, $R$ and $S$ are taken in the eigenspaces as in \Cref{lem:eigenspace fields}, so that they are defined (possibly on a twist) over $\Fpk$, where $k$ is the embedding degree of $n$, and all pairing values lie in $\mu_n \subseteq \Fpk^\times$.} and suppose that
  $\zeta_1 = e_n(P, R)$ and $\zeta_2 = e_n(Q^\rho, S)$ have orders $k_1$ and $k_2$ with
  \(
  k_1 k_2 > 8 \gcd(n,2)\cdot(p^3 + p).
  \)
  Then $A$ is supersingular of type $b$.

    Conversely, let $A$ be supersingular of type $b \in \{-1, 0, 1\}$ and $p > 2^7$, and let
  $P \in A[\pi - 1]$, $Q^\rho \in A[\pi + 1]$, $R \in A[\pi - p]$ and $S \in A[\pi + p]$ be generators
  of the four eigenspaces. Then these points satisfy the above condition.
\end{corollary}

\begin{proof}
We have
$k_1 \mid \mathrm{ord}(P)$ and $k_2 \mid \mathrm{ord}(Q)$, so the subgroup $H = \langle P, Q^\rho \rangle \subseteq A(\Fpt)$ satisfies
\[
  |H| = \frac{\mathrm{ord}(P)\,\mathrm{ord}(Q)}{|\langle P \rangle \cap \langle Q^\rho \rangle|}
      \geq \frac{k_1 k_2}{\gcd(2, n)} > 8(p^3 + p),
\]
which is the width of the Hasse--Weil interval over $\Fpt$. Moreover, $|H|$ divides $n^2$, so~\Cref{lem:width-bound-argument} gives $|A(\Fpt)| = n^2$. As $A$ is defined over
$\Fp$,~\Cref{thm:supersingular-orders} shows that $A$ is supersingular of type $b$.

For the converse, $A(\mathbb{F}_p)$ has exponent $n$. If $b = 0$, then $k_1 = k_2 = n/2$ by~\Cref{lem:b0}, and the condition
$n^2/4 > 16(p^3 + p)$ holds for $p \geq 67$. If
$b = \pm 1$, then $n$ is odd and $k_1, k_2 \geq n/3$ by~\Cref{lem:b1} and~\Cref{lem:bm1},
and $n^2/9 > 8(p^3 + p)$ holds for $p \geq 79$.
\end{proof}
\begin{remark}
    In \Cref{cor:conclusive-testing-cyclic}, the points $R$ and $S$ only serve as witnesses: a pairing $e_n(P, R)$ of order $k_1$ certifies that $P$ has order at least $k_1$, and similarly for $Q$. Supersingularity then follows from the width bound on $\langle P, Q^\rho\rangle$. Such witnesses cannot be found in $A(\Fpt)$, as pairings between points of $A(\Fpt)$ take values in $\mu_n \cap \Fpt^\times = \mu_{\gcd(n, p^2 - 1)}$, which has order at most 3. This is why we pass to $\Fpk$, where non-degeneracy of $e_n$ guarantees that suitable witnesses exist.
\end{remark}

\paragraph{Proof of Superspeciality.}
The above results give a proof on the type $b$ of $\J$, and therefore also on $\phi_\J(x)$.
Again, following \Cref{lem:poonen en co}, whenever $b = \pm 2$ this immediately also proves that $\J$ is superspecial,
and otherwise, following \Cref{thm:twisting superspecial},
we compute the automorphism group of the associated hyperelliptic curve.

\subsection{Conclusive testing of extremal \texorpdfstring{abelian surfaces}{abelian surfaces} over $\Fpt$}\label{sec:extremal-conclusive}

The pairing-based approach for abelian surfaces of type $b = \pm 2$ directly extends to
conclusive pairing-based superspeciality verification of minimal and maximal abelian surfaces over $\Fpt$.
However, in these situations, we should no longer assume $\Fp$-rationality of any points, and
hence no longer have the isotropy conditions on $\J(\Fp)$ and $\J^t(\Fp)$ to our advantage.
Given a to-be-verified basis $B = \{ P_1, \ldots, P_4 \}$ of $\J(\Fpt)$,
our approach is to use pairings to determine the size of $H = \langle P_1, \ldots, P_4\rangle$,
which implies minimality or maximality as soon as $|H|$ is larger than the width of the Hasse-Weil interval.
We describe a simple informal algorithm for maximal abelian surfaces,
and the variant for minimal abelian surfaces is self-evident.

\begin{lemalg}\label{alg:conclusive-extremal-surfaces}
  Let $A$ be an abelian surface over $\Fpt$, and $g \in \mu_{p+1} \subseteq \Fpt^*$
  a fixed generator of $\mu_{p+1}$. Repeat the following steps until termination.
  \begin{enumerate}[itemsep=0mm]
    \item {\upshape Sample 4 random points $P_1, \ldots, P_4 \rgets A(\Fpt)$}
    \item {\upshape If $[p+1]P_i \neq \zero_A$ for any $i$, return \textbf{false}}
    \item {\upshape Compute $\zeta_{i,j} = e_{p+1}(P_i, P_j)$ and set $Z = (\log_g(\zeta_{i,j}))$ as a matrix over $\Z$}
    \item {\upshape Compute the invariant factors $d_i$ of the Smith normal form of $Z$}
    \item {\upshape Set $N_Z= \prod_{i=1}^{4} \frac{p+1}{\gcd(d_i, p+1)}$. If $N_Z > 8p^3 + 8p$, return \textbf{true}. Otherwise, repeat.}
  \end{enumerate}
\end{lemalg}

\begin{lemma}
  \Cref{alg:conclusive-extremal-surfaces} returns \textbf{\upshape true} only for maximal supersingular
  abelian surfaces over $\Fpt$, and \textbf{\upshape false} otherwise.
\end{lemma}
\begin{proof}
  If $[p+1]P_i \neq \zero_A$ for any $i$, then $A$ is clearly non-maximal.
  Then, when the four points $P_i$ span a subgroup $H$
  with $|H| > 8p^3 + 8p$ by the width-bound argument on the Hasse--Weil interval, we
  have that $A$ must have order $(p+1)^4$, i.e., maximal. We can measure the size of $H$
  using the image of $Z$ as a linear map.
  This can be done effectively\footnote{
    See \url{https://math.stackexchange.com/questions/3252961/image-of-a-linear-transformation-with-a-matrix-in-mathbb-z-m}.
  }
  using the invariants $d_1, \ldots, d_4$ of the Smith normal form of $Z$ as a matrix over $\Z$:
  the image of $Z$ is then of size $\prod \frac{p+1}{\gcd(d_i, p+1)}$.
\end{proof}

\begin{remark}
  Given a prime $p$, it is theoretically possible to derive the expected number of iterations
  required in \Cref{alg:conclusive-extremal-surfaces}.
  Experimental data for primes of the form $p = 2^{k} \cdot 5 - 1$
  shows that we succeed on the first try with probability over $99 \%$ already from
  $k = 18$. Thus, in practical applications, \Cref{alg:conclusive-extremal-surfaces}
  works on the first try.
\end{remark}

\begin{remark}
  We have written \Cref{alg:conclusive-extremal-surfaces} with simplicity in mind,
  however, theoretically, as \Cref{alg:conclusive-extremal-surfaces} samples random points, it may never terminate
  as the algorithm could sample low-order points over and over with
  exponentially small probability.
  This may be resolved by deterministically sampling points that span an increasingly larger
  subgroup consistently.
  For example, if $A$ is a Jacobian $\J$, the deterministic sampling of divisors reduces to deterministically
  sampling points $(x_0, y_0)$ on the associated hyperelliptic curve $C$, which can be done by sampling
  $x_0 = \omega^n$ for a generator $\omega$ of $k^*$ and increasing $n$.
  However, in practice, such an algorithm is tedious to work with,
  and does not significantly speed-up validation in practical situations.
\end{remark}

\subsection{Conclusive Testing for Cryptographic Primes}
\label{sec:cryptographic-primes}
Isogeny-based cryptography often relies on efficient computations of $(2,2)$-isogenies, hence the base field $\Fp$ or $\Fpt$
is chosen so that a large power of $2$ divides $p-1$ or $p+1$, and one
then works with principally polarized abelian surfaces with group structure $(p - 1)^{k}$, resp. $(p + 1)^k$,
where $k = 2$ when working over $\Fp$, or $k = 4$ when working over $\Fpt$.
This allows one to efficiently compute long chains of $(2,2)$-isogenies using
rational maximally isotropic $(2^f, 2^f)$-subgroups, with $2^f \mid p \pm 1$.
This large amount of rational $2^\bullet$-torsion allows us to simplify conclusive testing significantly,
by a generalisation of a similar supersingularity test for elliptic curves~\cite[Alg. 2]{dCSIDH}.
The core idea is to use the $2^\bullet$-torsion instead of a $2^\bullet$-pairing
to show that both $A$ and $A^t$ have a certain group structure, no matter what the isotropy conditions
on this $2^\bullet$-torsion are.
In fact, as we drop the pairing, the results in this subsection do not require the principal
polarization either: they rely only on the orders of rational points and the width-bound argument of
\Cref{lem:width-bound-argument}.
We analyse these cases first over $\Fp$ and then over $\Fpt$,
to give a conclusive test of extremality, hence, superspeciality, hence supersingularity.

\paragraph{Over $\Fp$.}
We describe this in more detail for abelian surfaces of type $b = -2$, that is, assuming $A(\Fp)$ has order $(p+1)^2$.
If $P$ and $Q$ have order $2^f$ and $[2^{f-1}] P \neq [2^{f-1}] Q$, then
$\langle P, Q \rangle \cong \left(\ZZ/{2^f} \ZZ \right)^2$ gives us a subgroup of size $2^{2f}$ in $A(\Fp)$.
By the width-bound argument, when $2^{2f} >  8 \cdot \sqrt{p} \cdot (p+1)$, which is satisfied as soon
as $f > \tfrac{3}{2} + \tfrac{3}{4} \log_2 p$, the order of $A(\Fp)$ must be $(p+1)^2$.
The next lemma generalizes this argument when relaxing the assumption that $P, Q$ have full order $2^f$.

\begin{lemma}\label{lem:two torsion test}
  Let $A$ be an abelian surface over $\Fp$ and let $f$ be the largest integer such that $2^f \mid p + 1$.
  Let $P, Q \in A[2^f](\Fp)$ of order $2^{k_P}$ and $2^{k_Q}$ respectively, with $k_P \geq k_Q$.
  If $[2^{k_P - 1}]P \neq [2^{k_Q-1}] Q$ and $k_P + k_Q > 3 + \tfrac{3}{2} \log_2 p$,
  then $A$ is supersingular.
\end{lemma}

\begin{proof}
  Since $[2^{k_P-1}]P \neq [2^{k_Q-1}] Q$, we have that $\langle P, Q \rangle \cong \ZZ_{2^{k_P}} \times \ZZ_{2^{k_Q}} \hookrightarrow A(\Fp)[p+1]$.
  Let $A[p+1](\Fp) \cong \Z_{h_0} \times \Z_{h_1} \times \Z_{h_2} \times \Z_{h_3}$ with $h_{i+1}\mid h_i$.
  Then we have that $h_0 \cdot h_1 \geq 2^{k_P + k_Q} > 8 \sqrt{p} (p+1)$. The argument in the proof of~\Cref{lem:prob-p+1} shows that $a_p=0$ and $|A(\Fp)| = (p+1)^2$, so $A$ is supersingular of type $b=-2$.
\end{proof}
Finding such points $P$ and $Q$ is simple:
randomly sampling $P, Q \rgets A(\Fp)$, and using $[h]P, [h]Q$ where $h = \tfrac{p + 1}{2^f}$ almost always works on the first try,
as such points almost always have order $2^f$ or $2^g$ with $g$ only a bit smaller than $f$.
Thus, for $f$ close to $\log_2 p$, the condition on $k_Q$ is almost always satisfied.
Interestingly, one can also sample directly in $A(\Fp) \setminus [2] A(\Fp)$,
see e.g.~\cite{rotk,tate-profile}, to enforce $k_Q = f$ or $k_Q = f - 1$, depending
on the exponent of $A$.
In total, this
gives a highly efficient conclusive proof of supersingularity
using only simple and efficient $\Fp$-arithmetic for primes such that $2^f \mid p+1$ and $f\approx \log_2 p $.
The case $b = 2$ follows similarly as \Cref{lem:two torsion test},
by repeating the above procedure for both $A(\Fp)$ and $A^t(\Fp)$ with $2^f \mid p-1$.

\paragraph{Over $\Fpt$.}
We again explain the argument only for the maximal case where $A(\Fpt)$ has order $(p+1)^4$,
and omit the analogous minimal case.

\begin{lemma}\label{lem:fpt two torsion}
  Let $A$ be an abelian surface over $\Fpt$, let $f$ be the largest integer such that $2^f \mid p + 1$
  and let $P_1, P_2, P_3, P_4\in A[p+1]$ be such that $Q_i \colonequals [\tfrac{p+1}{2^f}] P_i$ has order $2^{k_i}$ for $1 \leq i \leq 4$, and the points $[2^{k_i-1}]Q_i$ span $A(\Fpt)[2]$. If $\sum k_i > 3 + 3\log_2 p$,
  then $A$ is maximal, hence superspecial.
\end{lemma}
\begin{proof}
  As for \Cref{lem:two torsion test}, now requiring $2^{\sum k_i} > 8 p (p^2 + 1)$.
\end{proof}
Sampling such a generating set $\{P_1, \ldots, P_4\}$ can be done efficiently but is rather technical, see~\cite[Thm. 1]{tate-profile}.
Nevertheless, similar to the $\Fp$-case, sampling random points $P_1, \ldots, P_4$ is highly likely to succeed on a first try,
as long as the $2$-adic order of each of $P_1, \ldots P_4$ exceeds $\tfrac{3}{4}(1 + \log_2 p)$,
and the $2$-torsion points below $P_i$ span $A[2]$, so that they are independent.

\section{Benchmarking Results}
\label{sec:results}
We now compare the performance of our algorithms, which
we have implemented in \texttt{Magma}~\cite{magma}, to existing methods in the literature to verify supersingularity.
We are not aware of other probabilistic methods to verify supersingularity.
We therefore compare both \IsSupersingular as well as the pairing-based conclusive methods
to point counting on $A$. For superspeciality, we compare to the computation of the Cartier-Manin matrix (via \verb|CartierRepresentation|)
in \texttt{Magma}. Lastly, we also compare the speed-up achieved for
primes of the form $p = c \cdot 2^f - 1$ for Jacobians of type $b = -2$,
using the $2^\bullet$-torsion test from \cref{lem:two torsion test}, as this situation
is common in isogeny-based cryptography.

\paragraph{Performance of probabilistic approaches.}
For superspeciality, we track the average computation time to probabilistically verify superspeciality of 100 supersingular Jacobians
over $\Fp$ for primes $p$ from $29$ up to $211$
with \Cref{alg:superspecial-test}.
We compare this to the computation of the Cartier--Manin matrix in \texttt{Magma}.
For supersingularity,
we track the average computation time to probabilistically verify supersingularity of 1000 supersingular Jacobians
over $\Fp$, for primes $p$ from $29$ up to $211$
with \IsSupersingular (\Cref{alg:supersingularity_test}).
We compare this to a point counting approach based on \Cref{lem: the twist is enough} in \texttt{Magma}.
The results are given in \cref{fig:prob-testing}.

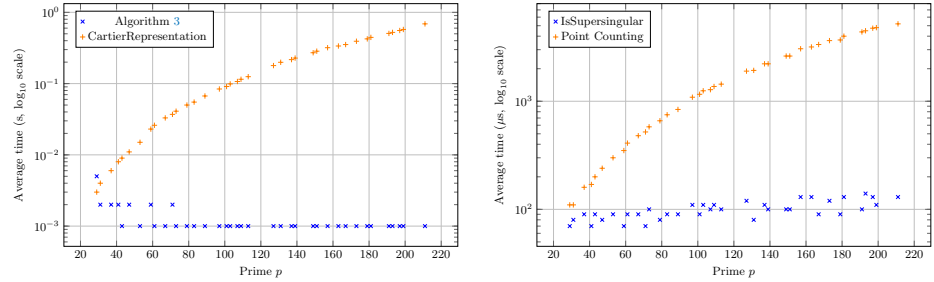
\begin{figure}[h!]
    \centering
    \begin{subfigure}[t]{0.48\textwidth}
        \centering
        \begin{tikzpicture}[scale=.5]
            \begin{axis}[
            xlabel={Prime $p$},
            ylabel={Average time (s, $\log_{10}$ scale)},
            grid=major,
            ymode=log,
            log basis y=10,
            legend pos=north west,
            legend image code/.code={
                \draw[#1] plot coordinates {(0.3cm,0cm)};
            },
            width=12cm,
            height=8cm
            ]

            \addplot[color=blue, mark=x]
                table[x=p, y=PointMult_Avg] {figs/benchmark_data_sspecial_prob.dat};
            \addlegendentry{Algorithm \ref{alg:superspecial-test}}

            \addplot[color=orange, mark=+]
                table[x=p, y=Cartier_Avg] {figs/benchmark_data_sspecial_prob.dat};
            \addlegendentry{CartierRepresentation}

            \end{axis}
        \end{tikzpicture}
        \vspace*{-0.5em}
        \caption{Superspeciality: \Cref{alg:superspecial-test} vs. Cartier–Manin matrix computation.}
        \label{fig:prob-testing-superspecial}
    \end{subfigure}
    \hfill
    \begin{subfigure}[t]{0.48\textwidth}
        \centering
        \begin{tikzpicture}[scale=.5]
            \begin{axis}[
            xlabel={Prime $p$},
            ylabel={Average time ($\mu$s, $\log_{10}$ scale)},
            ymode=log,
            log basis y=10,
            grid=major,
            legend pos=north west,
            legend image code/.code={
                \draw[#1] plot coordinates {(0.3cm,0cm)};
            },
            width=12cm,
            height=8cm
            ]

            \addplot[color=blue, mark=x]
                table[x=p, y=PointMult_Avg] {figs/benchmark_data_ssing_prob.dat};
            \addlegendentry{IsSupersingular}

\addplot[color=orange, mark=+]
                table[x=p, y=PointCount_Avg] {figs/benchmark_data_ssing_prob.dat};
            \addlegendentry{Point Counting}

            \addlegendentry{Cartier Representation}

            \end{axis}
        \end{tikzpicture}
        \vspace*{-0.5em}
        \caption{Supersingularity: \IsSupersingular vs. point counting.}
        \label{fig:prob-testing-supersingular}
    \end{subfigure}
    \caption{Performance of probabilistic verification: average time to test 100 Jacobians for superspeciality (left), 1000 for supersingularity (right).}
    \label{fig:prob-testing}
\end{figure}

\paragraph{Performance of conclusive approaches.}
For supersingularity, we track the average total computation time to conclusively verify supersingularity of 100 supersingular Jacobians of type $b=-2$
over $\Fp$, for primes $p$ of the form $p = c2^f -1$ with $c \in \{5, 3 \cdot 5, 3^2 \cdot 5, 3\cdot 5\cdot 7\}$.
We first compare the pairing-based methods from \Cref{sec:conclusive} to the point counting approach based on Lemma \ref{lem: the twist is enough} in \texttt{Magma},
for $p$ from $29$ up to $188743679 \approx 2^{27}$ ($f$ between $2$ and $22$).
We then show the improvement of the $2^\bullet$-torsion test (\Cref{lem:two torsion test})
over the conclusive test from \Cref{sec:conclusive} in \Cref{fig:conc-testing}, for $p$ from $2^{13}$ up to $2^{297}$ ($f$ from $9$ up to $293$).

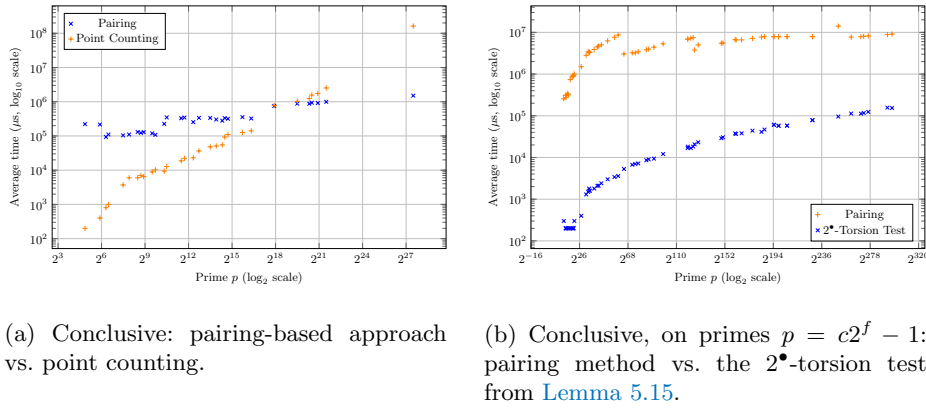
\begin{figure}[h!]
    \centering
\begin{subfigure}[t]{0.48\textwidth}
        \centering
        \begin{tikzpicture}[scale=.5]
            \begin{axis}[
            xlabel={Prime $p$ ($\log_2$ scale)},
            ylabel={Average time ($\mu$s, $\log_{10}$ scale)},
            xmode=log,
            log basis x=2,
            ymode=log,
            log basis y=10,
            grid=major,
            legend pos=north west,
            legend image code/.code={
                \draw[#1] plot coordinates {(0.3cm,0cm)};
            },
            width=12cm,
            height=8cm
            ]

            \addplot[color=blue, mark=x]
                table[x=p, y=Pairing_Avg] {figs/benchmark_data_conc_big_p_together.dat};
            \addlegendentry{Pairing}

            \addplot[color=orange, mark=+]
                table[x=p, y=PointCount_Avg] {figs/benchmark_data_conc_big_p_together.dat};
            \addlegendentry{Point Counting}

            \end{axis}
        \end{tikzpicture}
        \vspace*{-0.5em}
        \caption{Conclusive: pairing-based approach vs. point counting.}
        \label{fig:conc-testing-supersingular}
    \end{subfigure}
    \hfill
    \begin{subfigure}[t]{0.48\textwidth}
        \centering
        \begin{tikzpicture}[scale=.5]
            \begin{axis}[
            xlabel={Prime $p$ ($\log_2$ scale)},
            ylabel={Average time ($\mu$s, $\log_{10}$ scale)},
            xmode=log,
            log basis x=2,
            ymode=log,
            log basis y=10,
            grid=major,
            legend pos=south east,
            legend image code/.code={
                \draw[#1] plot coordinates {(0.3cm,0cm)};
            },
            width=12cm,
            height=8cm
            ]

            \addplot[color=orange, mark=+]
                table[x=p, y=Pairing_Avg] {figs/benchmark_data_conc_crypto.dat};
            \addlegendentry{Pairing}

            \addplot[color=blue, mark=x]
                table[x=p, y=PairingCrypto_Avg] {figs/benchmark_data_conc_crypto.dat};
            \addlegendentry{$2^\bullet$-Torsion Test}

            \end{axis}
        \end{tikzpicture}
        \vspace*{-0.5em}
        \caption{Conclusive, on primes $p = c2^f -1$: pairing method vs. the $2^\bullet$-torsion test from \Cref{lem:two torsion test}.}
        \label{fig:two-tors-test-improvement}
    \end{subfigure}
    \caption{Performance of conclusive verification: average time to test 100 Jacobians.}
    \label{fig:conc-testing}
\end{figure}

\section{Future Work}
\label{sec:future}
This work covers probabilistic and conclusive supersingularity
and superspeciality testing in dimension 2 over $\Fp$.
A natural direction for follow-up work is to determine similar efficient
criteria for supersingular abelian varieties of dimension $g > 2$, which
comes down to a specific case of the Ekedahl--Oort stratification of particular
interest to cryptography. Similarly, the case of abelian surfaces over $\Fpt$
has only been partially answered by this work, as the results in
\Cref{sec:testing,sec:conclusive} translate easily to $\Fpt$
for the Jacobians of type $b \in \{ \pm 1, \pm 2 \}$.
The exploration of supersingularity verification for
abelian surfaces with other Weil polynomials is left for future work.

Furthermore, the following specializations of the approaches in this work are possible.

\paragraph{Kummer surfaces.}
When using Kummer surfaces, we do not distinguish $P$ and $-P$.
This means we must specialize certain results in this work to such a specific setting,
as some results require this precision. For example,
we may find ordinary abelian surfaces with
$|\J(\Fp)| = (p + 1)^2$ and $|\J^t(\Fp)| = (p-1)(p + 1)$.
This does not seem too problematic,
but a more careful examination of such cases is required to claim with certainty that
our probabilistic tests remain convincing. Note that the conclusive tests from \Cref{sec:conclusive}
remain applicable, as such pairings directly verify the \emph{group structure} instead of the order.

\paragraph{Optimized arithmetic.}
In this work, we have refrained from a detailed analysis of the cost of scalar multiplication and pairings on
Jacobians, and instead have simply written $[n]P$ or $e_n(P, Q)$.
In practice, one may apply well-known techniques, such as differential addition chains
or cubical arithmetic, to determine a precise cost of our algorithms in terms
of finite field operations, as has been done in the
case of elliptic curves~\cite{val-and-the-gang,pairings-CSIDH,dCSIDH,simplepairings25}.

\paragraph{Non-ordinary \texorpdfstring{abelian surfaces}{abelian surfaces}.}
To complete the efficient Ekedahl-Oort classification for abelian surfaces over $\Fp$,
one may study the case of non-ordinary abelian surfaces, i.e., $p$-rank 1. As such abelian surfaces have similarly
specific Weil polynomials, similar methods as proposed in this work should apply, although
with different technical derivations and pairing properties.

\printbibliography

\newpage
\appendix
\crefalias{section}{appendix}
\crefalias{subsection}{appendix}
\section{Non-trivial Automorphisms of Jacobians}\label{app:automorphisms}
In general, automorphisms of two-dimensional Jacobians are induced by automorphisms of the associated hyperelliptic
curves of genus 2. Any Jacobian will have the automorphism $[-1]$,
and one therefore often analyzes the \emph{reduced automorphism group} $\RA(C)$, resp., $\RA(\J)$.

For dimension 2, the non-trivial reduced automorphism groups are well-understood. An overview of these is given by
\textcite{atlas}, and summarized in~\cite[App. A]{supersingular-digraphs}. We follow the latter to describe
those non-trivial automorphisms used in this work: $\tau$, $\rho$, $\omega$, $\zeta^{(5)}$, and $\zeta^{(6)}$.
We denote a primitive $n$-th root of unity in $\Fq^*$ by $\zeta_n$.

\begin{itemize}
    \item The automorphism $\tau$ can be described as $x \mapsto 1/x$ on a suitable model for $C$, which always exists. Its image in $\RA(C)$ has order 2, and its lifts to $\Aut(C)$ satisfy $\tau^2 = \iota$, the hyperelliptic involution.
    \item The automorphism $\rho$ can be described as $x \mapsto \tfrac{(2 \zeta_3 + 1)(v^2 - 1)x + 3(v+1)^2}{3(v-1)^2x + (2 \zeta_3 + 1)(v^2 - 1)}$ for some $v \in \Fq$ on a suitable model for $C$,
    which always exists. It has order 3.
    \item The automorphism $\omega$ can be described as $x \mapsto \tfrac{x - (\sqrt{2} + 1)}{(\sqrt{2} - 1)x + 1}$ on a suitable model for $C$,
    which always exists. It has order 4.
    \item The automorphism $\zeta^{(5)}$ can be described as $x \mapsto \zeta_5 x$
    on the curve $C : y^2 = x^5 - 1$. It has order 5.
    \item The automorphism $\zeta^{(6)}$ can be described as $x \mapsto \zeta_6 x$
    on the curve $C : y^2 = x^6 + 1$. It has order 6.
\end{itemize}

\section{Classification of Superspecial Jacobians over $\Fpt$}\label{app:twisted superspecials}

We may repeat a similar computation as
we performed in \Cref{sec:superspecial criterion}
for Jacobians over $\Fpt$.
Namely, given a maximal, resp., minimal, Jacobian, we may compute the characteristic polynomial
of a non-quadratic twist $C^\sigma$ for a non-trivial automorphism $\sigma \in \RA(C)$.
This gives us the list of characteristic polynomials of twisted superspecial curves, resp., Jacobians, over $\Fpt$.
Similar to the case over $\Fp$, we may use the results in \cite{atlas,supersingular-digraphs} to determine the partial $\RA(C)$ from these
results, or a similar characterization of such Jacobians in terms of the associated $(2,2)$-isogenous product of elliptic curves.
Over $\Fpt$, we write that $\J$ is of type $(a,b)$ if
\[
    \phi_\J(x) = x^4 - a p x^3 + b p^2 x^2 - a p^3 x + p^4.
\]
The quadratic twist maps $(a, b) \mapsto (-a, b)$, hence we may classify these up to $\pm a$.

\begin{theorem}\label{thm:twisting superspecial2}
	Let $\J$ be a Jacobian over $\Fpt$. Then $\J$ is superspecial if and only if
	\begin{itemize}
		\item $\J$ is minimal or maximal, or
		\item $\J/\Fpt$ is a twisted superspecial
        \begin{enumerate}
            \item of type $(0, -2)$, e.g. $\phi_\J(x) = (x - p)^2(x + p)^2$, and $\sigma \in \RA(C)$ (whose lift in $\Aut(C)$ is an involution not equal to the hyperelliptic involution).
            \item of type $(\pm 2, 3)$, e.g. $\phi_\J(x) = (x^2 \pm px + p^2)^2$, and $\rho \in \RA(C)$
            \item of type $(0, 2)$, e.g. $\phi_\J(x) = (x^2 + p^2)^2$, and $\tau \in \RA(C)$
            \item of type $(\pm 1, 1)$, e.g. $\phi_\J(x) = (x^4 \pm px^3 + p^2x^2 \pm p^3x + p^4)$, and $\zeta^{(5)} \in \RA(C)$
            \item of type $(0, 1)$, e.g. $\phi_\J(x) = (x^2 - px + p^2)(x^2 + px + p^2)$, and $\zeta^{(6)} \in \RA(C)$
            \item of type $(0, 0)$, e.g. $\phi_\J(x) = (x^4 + p^4)$, and $\omega \in \RA(C)$
            \end{enumerate}
	\end{itemize}
\end{theorem}

\paragraph{Superspeciality Verification.}
With the above result, superspeciality verification over $\Fpt$ is surprisingly easy:
either check that $\J$ is minimal or maximal using the results from \Cref{sec:testing}, or compute
the automorphism group of the associated hyperelliptic curve, and find a minimal or maximal twist.

\paragraph{Lifting $\Fp$ Curves.}
Recall that a curve over $\Fp$ with characteristic polynomial $\phi(x) = \prod (x - \alpha_i)$ lifts to a curve
over $\Fpt$ with characteristic polynomial $\phi(x) = \prod (x- \alpha_i^2)$.
This gives us that curves of type $b = \pm 2$ over $\Fp$ become of type $(\pm4, 6)$ over $\Fpt$,
curves of type $b = \pm 1$ become type $(\pm 2, 3)$,
and curves of type $0$ become type $(0, 2)$.

\end{document}